\documentclass[12pt]{amsart}
\usepackage{amscd,amsmath,amsthm,amssymb}
\usepackage{mathtools}
\usepackage[margin=2cm]{geometry}
\usepackage{epsfig}
\usepackage{rawfonts}
\usepackage{enumerate}
\usepackage{graphics}
\usepackage{multirow}
\usepackage{xspace}
\usepackage[ruled,vlined]{algorithm2e}
\SetKwInput{KwInput}{Input}
\SetKwInput{KwOutput}{Output}
\usepackage{graphicx}
\usepackage{mathrsfs}
\usepackage{amsmath}
\usepackage{amsfonts}
\usepackage{amssymb}
\usepackage{amsthm}
\usepackage{graphicx}
\usepackage{booktabs}
\usepackage{caption}
\usepackage{listings}
\usepackage{setspace}
\usepackage[mathscr]{eucal}
\usepackage{pgfplots}
\usepackage{hyperref}
\usepackage{wrapfig}
\usepackage{floatflt,epsfig}
\usepackage{ dsfont }
\usepackage{amscd}
\usepackage{tikz-cd}
\usepackage{fancyhdr}
\usepackage[all]{xy}
\usepackage{latexsym}
\usepackage{amscd}
\usepackage{pifont}
\usepackage{subfig}
\usepackage{easyReview}
\usepackage{subfig}
\usepackage{pstricks-add}
\usepackage{pgf,tikz,pgfplots}
\usetikzlibrary{decorations.pathreplacing}
\pgfplotsset{compat=1.15}
\usepackage{mathrsfs}
\usetikzlibrary{arrows}
\usetikzlibrary[patterns]
 
\def\NZQ{\Bbb}               

\def\ZZ{{\NZQ Z}}

\newcommand{\cA}{\mathcal{A}}
\newcommand{\cB}{\mathcal{B}}

\newcommand{\cP}{\mathcal{P}}
\newcommand{\Pc}{\mathcal{P}}

\newcommand{\cM}{\mathcal{M}}

\newcommand{\cU}{\mathcal{U}}
\newcommand{\cQ}{\mathcal{Q}}

\newcommand{\cR}{\mathcal{R}}
\newcommand{\Rc}{\mathcal{R}}
\newcommand{\cS}{\mathcal{S}}
\newcommand{\cT}{\mathcal{T}}

\newcommand{\fR}{\mathfrak{R}}
\newcommand{\fC}{\mathfrak{C}}

\renewcommand{\qedsymbol}{$\square$}

\def\opn#1#2{\def#1{\operatorname{#2}}} 
\opn\chara{char} \opn\length{\ell} \opn\pd{pd} \opn\rk{rk}
\opn\projdim{proj\,dim} \opn\injdim{inj\,dim} \opn\rank{rank}
\opn\depth{depth} \opn\grade{grade} \opn\height{height}
\opn\embdim{emb\,dim} \opn\codim{codim}

\opn\Tr{Tr} \opn\bigrank{big\,rank}
\opn\superheight{superheight}\opn\lcm{lcm}
\opn\trdeg{tr\,deg}
	\opn\reg{reg} \opn\lreg{lreg} \opn\ini{in} \opn\lpd{lpd}
	\opn\size{size} \opn\sdepth{sdepth}
	\opn\link{link}\opn\fdepth{fdepth}\opn\lex{lex}\opn\dist{dist}
	\opn\div{div} \opn\Div{Div} \opn\cl{cl} \opn\Cl{Cl}
	\opn\Spec{Spec} \opn\Supp{Supp} \opn\supp{supp} \opn\Sing{Sing}
	\opn\Ass{Ass} \opn\Min{Min}\opn\Mon{Mon}
	\opn\Ann{Ann} \opn\Rad{Rad} \opn\Soc{Soc}
	\opn\Im{Im} \opn\Ker{Ker} \opn\Coker{Coker} \opn\Am{Am}
	\opn\Hom{Hom} \opn\Tor{Tor} \opn\Ext{Ext} \opn\End{End}
	\opn\Aut{Aut} \opn\id{id}
	
	\opn\nat{nat}
	\opn\pff{pf}
	\opn\Pf{Pf} \opn\GL{GL} \opn\SL{SL} \opn\mod{mod} \opn\ord{ord}
	\opn\Gin{Gin} \opn\Hilb{Hilb}\opn\sort{sort}
	\opn\aff{aff} \opn
	\con{conv} \opn\relint{relint} \opn\st{st}
	\opn\lk{lk} \opn\cn{cn} \opn\core{core} \opn\vol{vol}
	\opn\link{link} \opn\star{star}\opn\lex{lex}\opn\set{set}
	\opn\gr{gr}
	
	\def\pot#1#2{#1[\kern-0.28ex[#2]\kern-0.28ex]}

	\opn\dirlim{\underrightarrow{\lim}}
	\opn\inivlim{\underleftarrow{\lim}}
	\let\to=\rightarrow
	
	\def\Implies{\ifmmode\Longrightarrow \else
		\unskip${}\Longrightarrow{}$\ignorespaces\fi}
	\def\implies{\ifmmode\Rightarrow \else
		\unskip${}\Rightarrow{}$\ignorespaces\fi}
	\def\iff{\ifmmode\Longleftrightarrow \else
		\unskip${}\Longleftrightarrow{}$\ignorespaces\fi}

	\let\:=\colon
	\let\epsilon\varepsilon
	\let\kappa=\varkappa
	\def\qed{\ifhmode\textqed\fi
		\ifmmode\ifinner\quad\qedsymbol\else\dispqed\fi\fi}
	\def\textqed{\unskip\nobreak\penalty50
		\hskip2em\hbox{}\nobreak\hfil\qedsymbol
		\parfillskip=0pt \finalhyphendemerits=0}
	\def\dispqed{\rlap{\qquad\qedsymbol}}
	
	\opn\dis{dis}
	\def\pnt{{\raise0.5mm\hbox{\large\bf.}}}
	
	\opn\Lex{Lex}
	
        \newtheorem{Theorem}{Theorem}[section]
	\newtheorem{Lemma}[Theorem]{Lemma}
	\newtheorem{Corollary}[Theorem]{Corollary}
	\newtheorem{Proposition}[Theorem]{Proposition}
	\newtheorem{Remark}[Theorem]{Remark}
	
	\newtheorem{Example}[Theorem]{Example}
	
	\newtheorem{Definition}[Theorem]{Definition}
	
	\newtheorem{Conjecture}[Theorem]{Conjecture}
	
    \newtheorem{Discussion}[Theorem]{Discussion}
    
    \newtheorem{Construction}[Theorem]{Construction}

\title[Palindromic switching rook polynomials and
domino-stable collections of cells]{Palindromic switching rook polynomials and
domino-stable collections of cells}

\author[Francesco Navarra, Ayesha Asloob Qureshi, Giancarlo Rinaldo]{
Francesco Navarra$^{a,b,\ast}$,
Ayesha Asloob Qureshi$^{a,c}$,
Giancarlo Rinaldo$^{d,e}$
}

\begin{document}
\makeatletter
\makeatother
\maketitle

\begin{center}
{\footnotesize

$^{a}$ Sabanci University, Faculty of Engineering and Natural Sciences,\\
Orta Mahalle, Tuzla 34956, Istanbul, Turkey

\vspace{0.15cm}

$^{b}$ \texttt{francesco.navarra@sabanciuniv.edu}

\vspace{0.1cm}

$^{c}$ \texttt{aqureshi@sabanciuniv.edu}, \texttt{ayesha.asloob@sabanciuniv.edu}

\vspace{0.15cm}

$^{d}$ Department of Mathematics and Computer Sciences, Physics and Earth Sciences,\\
University of Messina, Viale Ferdinando Stagno d'Alcontres 31,\\
98166 Messina, Italy

\vspace{0.15cm}

$^{e}$ \texttt{giancarlo.rinaldo@unime.it}

\vspace{0.35cm}

$^{\ast}$ Corresponding author

}
\end{center}

\vspace{0.4cm}


\vspace{0.15cm}


   \begin{abstract}
 Let $\cP$ be a finite collection of cells. Two non-attacking rook
 configurations on $\cP$ are switching-equivalent if one can be obtained
 from the other by a sequence of switches inside rectangles contained in
 $\cP$. The switching rook polynomial of $\cP$ enumerates the resulting
 equivalence classes according to the number of rooks.
 We characterize the collections of cells whose switching rook polynomial
 is palindromic. More precisely, we introduce a new combinatorial property, called \textit{domino-stability}, and we prove that the switching rook polynomial of $\cP$ is palindromic if and only if $\cP$ is domino-stable. 
\end{abstract}


\section*{Introduction}

The classical rook problem asks for the number of ways to place
non-attacking rooks on a board, where no two rooks are allowed to lie in
the same row or column. The corresponding generating function is called
the \emph{rook polynomial}. Rook placements on skew diagrams are closely
related to the enumeration of permutations with forbidden positions, a
connection first investigated in \cite{KR} and subsequently developed in
\cite{R}; see also \cite[Chapter~2]{S1} for a comprehensive account.

Several variants of the rook polynomial have been introduced over the
years; see, for instance, \cite{Ge}. In this paper, we study the
\emph{switching rook polynomial} of a collection of cells. Let $\cP$ be
a collection of cells, regarded as a pruned chessboard. A rook
configuration on $\cP$ is non-attacking if no two of its rooks are
connected by a horizontal or vertical path of edge-adjacent cells of
$\cP$. Two configurations of $k$ non-attacking rooks are said to be
\emph{switching-equivalent} if one can be obtained from the other by a
finite sequence of switches. Each switch replaces two rooks occupying
one pair of opposite corner cells of a rectangle contained in $\cP$ by
two rooks occupying the other pair of opposite corner cells.

For every $k\in\{1\ldots,r(\cP)\}$, let $\widetilde{r}_k(\cP)$ denote the
number of switching classes of $k$-rook configurations on $\cP$ and $\widetilde{r}_0(\cP)=1$, where
$r(\cP)$ is the maximum number of non-attacking rooks that can be placed
on $\cP$. The \emph{switching rook polynomial} of $\cP$ is
\[
\widetilde{r}_{\cP}(t)
=
\sum_{k=0}^{r(\cP)}\widetilde{r}_k(\cP)t^k,
\]

The main problem addressed in this paper is to determine when
$\widetilde{r}_{\cP}(t)$ is palindromic. Recall that a polynomial $p(t)=p_0+p_1t+\cdots+p_dt^d$
is \emph{palindromic} if $p_k=p_{d-k}$ for every
$k\in\{0,\ldots,d\}$. To answer this question, we introduce a new
combinatorial property of collections of cells, called
\emph{domino-stability},
and we prove the following result.

\begin{Theorem}[Theorem~\ref{Thm: main palindromicity}]
Let $\cP$ be a collection of cells. Then the switching rook polynomial
$\widetilde{r}_{\cP}(t)$ is palindromic if and only if $\cP$ is
domino-stable.
\end{Theorem}


For skew diagrams, the switching rook polynomial is also referred to
as the \emph{nested rook polynomial} (see \cite{AJ}); in this setting, our
characterization recovers \cite[Corollary~5.16]{AJ}. Moreover, for
simple thin polyominoes and parallelogram polyominoes,
domino-stability is equivalent to the $S$-property. Thus, our theorem
recovers the corresponding characterizations in
\cite[Theorem~4.2]{RR} and \cite[Theorem~4.10]{QRR}, respectively.

The proof of the sufficient direction of
Theorem~\ref{Thm: main palindromicity} is constructive. Let $\cP$ be domino-stable and set $d=r(\cP)$. We first
construct a $d$-rook configuration $\cT=\{T_1,\ldots,T_d\}$
which is unique up to switches and has exactly one rook in every row and
every column of $\cP$. This configuration provides a reference system
for all the other rook configurations. More precisely, if a rook of a
rook configuration $\cR$ lies in the row containing $T_a$ and in the column
containing $T_b$, we associate with it the arrow $a\longrightarrow b.$
Since the rooks of $\cR$ are non-attacking, the resulting directed
graph on $\{1,\ldots,d\}$ has indegree and outdegree at most one at
each vertex. Its connected components are therefore directed paths,
directed cycles, and isolated vertices. Using this decomposition,
for every $k\in\{0,\ldots,d\}$ we construct an explicit bijection
between the switching classes of $k$-rook configurations and those
of $(d-k)$-rook configurations.\\
For the converse, suppose that $\widetilde r_{\cP}(t)$ is
palindromic. Since $\widetilde r_0(\cP)=1$, palindromicity gives
$\widetilde r_d(\cP)=1$. Starting from a representative $\cT$ of the
unique switching class of $d$-rook configurations, we construct an injective map $\tau$ attaching to a rook a switching class of a $(d-1)$-rook configuration of $\cP$. We then prove that, if $\cP$ is not domino-stable, then $\tau$ is not surjective, that is, the number of switching classes of $(d-1)$-rook configurations is strictly greater than the number of cells of $\cP$, contradicting palindromicity.

The switching rook polynomial arises naturally in combinatorial
commutative algebra. Let $S_{\cP}
=
K[x_{\mathbf v}:\mathbf v \text{ is a vertex of } \cP]$ 
and let $I_{\cP}\subseteq S_{\cP}$ be the binomial ideal generated by
the inner $2$-minors of $\cP$. The quotient $K[\cP]=S_{\cP}/I_{\cP}$
is called the \emph{coordinate ring} of $\cP$. This construction was
introduced in \cite{Q} and later included in
\cite[Chapter~8]{HHO_Book}, as a generalization of the theory of
determinantal ideals of order two; see also the recent survey
\cite{Survey}.\\
The connection between rook theory and inner $2$-minor ideals was initially observed in \cite{EHQR}. So far, it is known that for several classes of collections of cells the switching rook polynomial of $\cP$ coincides with the $h$-polynomial of $K[\cP]$, denoted by $h_{K[\cP]}(t)$ (see
\cite{CNU3,DN2,JN,KV,NQR,QRR,RR,WG}). Therefore, whenever $h_{K[\cP]}(t)=\widetilde{r}_{\cP}(t)$,
our main theorem can be combined with Stanley's results on palindromic
$h$-polynomials \cite{S}. In particular, Gorensteinness of $K[\cP]$
implies that $\cP$ is domino-stable, and the converse holds when
$K[\cP]$ is a Cohen--Macaulay domain. Computational experiments
reported in \cite{N2} suggest that domino-stability may be sufficient
for the Gorenstein property even without the domain hypothesis.

The main theorem also yields a numerical criterion for
palindromicity. More precisely, let $d=r(\cP)$ and assume that $\widetilde r_d(\cP)=1.
$
Then
\[
\widetilde r_{\cP}(t)\text{ is palindromic}
\quad\Longleftrightarrow\quad
\widetilde r_1(\cP)=\widetilde r_{d-1}(\cP);
\]
see Corollary~
\ref{Corollary: numerical palindromicity criterion}. The numerical criterion stated above is formally parallel to a recent
result \cite[Theorem~1.1]{Hatasa}. More precisely, if
\[
h(K[G];t)=h_0+h_1t+\cdots+h_st^s
\]
is the $h$-polynomial of the edge ring of a bipartite graph $G$,
it is proved that the conditions $h_s=1$ and $h_1=h_{s-1}$
imply that $K[G]$ is Gorenstein and hence that $h(K[G];t)$ is
palindromic. Thus, both \cite[Theorem~1.1]{Hatasa} and
Corollary~\ref{Corollary: numerical palindromicity criterion} show
that, once the leading coefficient is equal to one, the equality of
the first and next-to-leading coefficients forces the symmetry of
all the coefficients. \\
The two results overlap in the case of weakly connected simple
collections of cells. Indeed, the associated bipartite graph $G(\cP)$ is weakly chordal and $I_{\cP}$ coincides with the
toric ideal defining its edge ring (see \cite[Proposition~3.2 and Theorem~3.3]{CNU1}), so \cite[Theorem~1.1]{Hatasa} and our criterion concern the same polynomial and
give compatible conclusions whenever $
h_{K[\cP]}(t)=\widetilde r_{\cP}(t)
$. On the other hand,
Corollary~\ref{Corollary: numerical palindromicity criterion} is
formulated entirely in terms of switching rook configurations and
applies to arbitrary collections of cells, without any
connectedness or hole-freeness assumption. In this sense, our
result extends to arbitrary collections of cells the same numerical
principle for palindromicity.

The paper is organized as follows. In
Section~\ref{Section: Preliminaries}, we recall the basic definitions
concerning collections of cells, rook configurations, and switching
equivalence. Section~\ref{Section: domino stable} introduces the
gluing of rectangles and the notion of domino-stability, and
establishes some properties of domino-stable collections
needed in the sequel.

In Section~\ref{Section3: sufficient}, we prove that
domino-stability implies palindromicity. After constructing the
unique switching class of maximum rook configurations, we use the
decomposition into directed paths and directed cycles described
above to obtain the bijections between the switching classes of
$k$-rook and $(r(\cP)-k)$-rook configurations.

Section~\ref{Section4: necessary} proves the converse implication.
We construct an injective map from one-rook classes to
$(r(\cP)-1)$-rook classes and show that this map cannot be
surjective when $\cP$ is not domino-stable. This completes the proof
of Theorem~\ref{Thm: main palindromicity} and yields the numerical
criterion of
Corollary~\ref{Corollary: numerical palindromicity criterion}.

Finally, Section~\ref{Section: last, main theorem} discusses the
algebraic consequences of the main theorem. For collections of cells
$\cP$ satisfying $h_{K[\cP]}(t)=\widetilde r_{\cP}(t),$
Corollary~\ref{Coro: Gorenstein} relates domino-stability to the
Gorenstein property of $K[\cP]$. We conclude with
Conjecture~\ref{Conj Gorenstein}, motivated by the computational
experiments in \cite{N2}.

\section{Collections of cells and switching rook polynomials}\label{Section: Preliminaries}

    In this section, we present the combinatorial preliminaries on collections of cells and provide the definition and basic properties of the switching rook polynomial. 
    
   \subsection{Collections of cells.}\label{SubSection Collection of cells} Let $(i,j),(k,l) \in \mathbb{Z}^2$. We define a partial order on $\mathbb{Z}^2$ by setting $(i,j) \leq (k,l)$ if and only if $i \leq k$ and $j \leq l$. Given $\mathbf{a} = (i,j)$ and $\mathbf{b} = (k,l)$ in $\mathbb{Z}^2$ with $\mathbf{a} \leq \mathbf{b}$, the set $[\mathbf{a},\mathbf{b}]=\{(m,n)\in \mathbb{Z}^2: i\leq m\leq k,\ j\leq n\leq l \}$ is called an \textit{interval} of $\mathbb{Z}^2$. If $i < k$ and $j < l$, the interval $[\mathbf{a},\mathbf{b}]$ is said to be \textit{proper}; in this case, we refer to $\mathbf{a}$ and $\mathbf{b}$ as the diagonal corners of $[\mathbf{a},\mathbf{b}]$, and to $\mathbf{c} = (i,l)$ and $\mathbf{d} = (k,j)$ as the anti-diagonal corners. More precisely, $\mathbf{a}$, $\mathbf{b}$, $\mathbf{c}$, and $\mathbf{d}$ are respectively called the \textit{lower left}, \textit{upper right}, \textit{upper left}, and \textit{lower right} \textit{corners} of $[\mathbf{a},\mathbf{b}]$. If $j=l$ (respectively, $i=k$), then $\mathbf{a}$ and $\mathbf{b}$ are in a \textit{horizontal} (respectively, \textit{vertical}) \textit{position}. We denote by $]\mathbf{a},\mathbf{b}[$ the set $\{(m,n)\in \mathbb{Z}^2: i< m< k,\ j< n< l\}$. A proper interval $C = [\mathbf{a},\mathbf{b}]$ with $\mathbf{b} = \mathbf{a} + (1,1)$ is called a \textit{cell} of $\mathbb{Z}^2$. If $C = [\mathbf{a},\mathbf{b}]$ is a cell with $\mathbf{c},\mathbf{d}$ anti-diagonal corners, then we denote the set of \textit{vertices} and \textit{edges} of $C$ by $V(C)=\{\mathbf{a},\mathbf{b},\mathbf{c},\mathbf{d}\}$ and $E(C)=\{\{\mathbf{a},\mathbf{c}\},\{\mathbf{c},\mathbf{b}\},\{\mathbf{b},\mathbf{d}\},\{\mathbf{a},\mathbf{d}\}\}$, respectively.
Let $\cP$ be a non-empty collection of cells in $\mathbb{Z}^2$. We define the vertex and edge sets of $\cP$ as $V(\cP)=\bigcup_{C\in \cP}V(C)$ and $E(\cP)=\bigcup_{C\in \cP}E(C)$, respectively. The \textit{rank} of $\cP$, denoted by $\vert \cP\vert$, is the number of cells in $\cP$. \\
We say that $\cP$ is a \textit{polyomino} if for any two distinct cells $C$ and $D$ in $\cP$ there exists a sequence of cells in $\cP$, namely $C=C_1,\dots,C_m=D$, such that $C_i\cap C_{i+1}$ is a common edge of $C_i$ and $C_{i+1}$ for all $i=1,\dots,m-1$. \\
Analogously, $\cP$ is said to be \textit{weakly connected} if for any two distinct cells $C$ and $D$ in $\cP$ there exists a sequence of cells in $\cP$, namely $C=C_1,\dots,C_m=D$, such that $C_i\cap C_{i+1}\neq \emptyset$ for all $i=1,\dots,m-1$. \\
A \textit{connected component} $\cP'$ of $\cP$ is a polyomino contained in $\cP$ that is maximal with respect to set inclusion; that is, if $A\in \cP\setminus \cP'$, then $\cP'\cup \{A\}$ is not a polyomino. For instance, in Figures \ref{Figure: Polyomino + weakly conn. coll. of cells + bounding rectangle} (A) and (B) we illustrate a polyomino and a weakly connected collection of cells with three connected components, respectively.
\begin{figure}[h]
\centering
\subfloat[]{\includegraphics[scale=0.5]{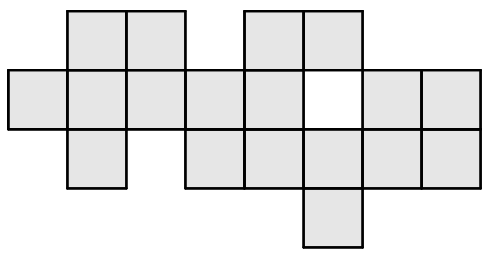}}\qquad
\subfloat[]{\includegraphics[scale=0.5]{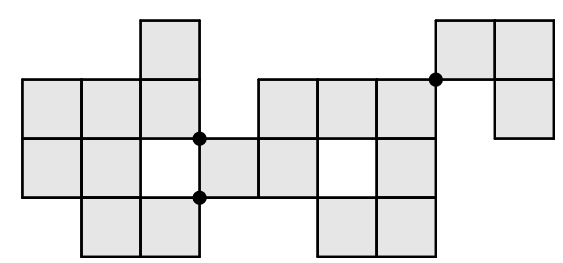}}\qquad
\subfloat[]{\includegraphics[scale=0.5]{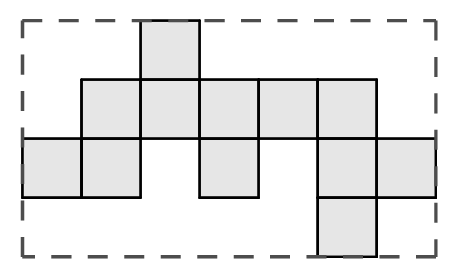}}
\caption{A polyomino, a weakly connected collection of cells, and a bounding rectangle of a polyomino.}
\label{Figure: Polyomino + weakly conn. coll. of cells + bounding rectangle}
\end{figure}
Consider two cells $A$ and $B$ of $\mathbb{Z}^2$ with $\mathbf{a}=(i,j)$ and $\mathbf{b}=(k,l)$ as the lower left corners of $A$ and $B$, respectively, with $\mathbf{a}\leq \mathbf{b}$. The \textit{cell interval} $[A,B]$, also called a \textit{rectangle}, is the set of cells of $\mathbb{Z}^2$ with lower left corner $(r,s)$ such that $i\leqslant r\leqslant k$ and $j\leqslant s\leqslant l$. If $(i,j)$ and $(k,l)$ are in horizontal (respectively, vertical) position, we say that the cells $A$ and $B$ are in a \textit{horizontal} (respectively, \textit{vertical}) \textit{position}.\\
The \textit{width} and the \textit{height} of $[A,B]$, denoted by $w([A,B])$ and $h([A,B])$, represent the number of horizontal and vertical cells of $[A,B]$, respectively. Specifically, $w([A,B])=k-i+1$ and $h([A,B])=l-j+1$. The \textit{size} of $[A,B]$ is defined as $h([A,B])w([A,B])$ and is denoted by $\mathrm{size}([A,B])$. We say that $[A,B]$ is \textit{horizontal} (respectively, \textit{vertical}) if $h([A,B])\leq w([A,B])$ (respectively, $h([A,B]) \geq w([A,B])$). If $h([A,B])=w([A,B])$, then $[A,B]$ is said to be a \textit{square}. Note that a square can be regarded as both a vertical and a horizontal rectangle. \\
If $\cP$ is a collection of cells, then we say that $[A,B]$ is the \textit{minimal bounding rectangle} of $\cP$ if $\cP\subseteq [A,B]$ and there does not exist a rectangle of $\mathbb{Z}^2$ containing $\cP$ that is properly contained in $[A,B]$ (Figure \ref{Figure: Polyomino + weakly conn. coll. of cells + bounding rectangle} (C)). A \textit{maximal rectangle} of $\cP$ is a rectangle $\cR$ of cells of $\cP$ such that no other rectangle whose cells all belong to $\cP$ properly contains $\cR$. \\
Consider two cells $A$ and $B$ of $\cP$ in a vertical (respectively, horizontal) position. The cell interval $[A,B]$ is called a \textit{column} (respectively, a \textit{row}) of $\cP$ if $A$ and $B$ are in vertical (respectively, horizontal) position, all the cells in $[A,B]$ belong to $\cP$, and there is no cell interval $[A',B']$ with $A'$ and $B'$ in vertical (respectively, horizontal) position such that $[A,B]\subsetneq [A',B']$. The polyomino in Figure~\ref{Figure: Polyomino + weakly conn. coll. of cells + bounding rectangle} (A) has eight maximal rectangles, nine columns, and seven rows. \\
Since an interval in $\mathbb{Z}^2$ uniquely corresponds to a cell interval in $\mathbb{Z}^2$, we associate with each interval $I$ of $\mathbb{Z}^2$ its corresponding rectangle, denoted by $\cR(I)$. If $I=[\mathbf{a},\mathbf{b}]$ and $\mathbf{c}$, $\mathbf{d}$ are the anti-diagonal corners of $I$, then the two cells of $\cR(I)$ containing $\mathbf{a}$ or $\mathbf{b}$ (respectively, $\mathbf{c}$ or $\mathbf{d}$) are called the \textit{diagonal} (respectively, \textit{anti-diagonal}) \textit{corner cells} of $\cR(I)$. We say that an interval $I$ is an \textit{inner interval} of a collection of cells $\cP$ if $\cR(I)\subseteq \cP$.
  
\subsection{Switching rook polynomial.} \label{Subsection Switching}We now introduce the terminology of non-attacking rooks on a collection of cells and the switching rook polynomial.

Let $\cP$ be a collection of cells. We say that two rooks $R_1$ and $R_2$ \textit{coincide} or \textit{overlap} if they are located in the same cell of $\cP$. Two rooks $R_1$ and $R_2$ are in \textit{attacking position} or \textit{attacking rooks} in $\cP$ if they do not overlap and there exists a row or a column of $\cP$ that contains the cells where $R_1$ and $R_2$ are placed. In contrast, two rooks are in \textit{non-attacking position} or \textit{non-attacking rooks} in $\cP$ if they do not overlap and they are not in attacking position. A \textit{$j$-rook configuration} in $\cP$ is a set of $j$ rooks arranged in non-attacking positions within $\cP$, where $j\geq 0$; for convention, the $0$-rook configuration is $\emptyset$. We say that a $j$-rook configuration in $\cP$ is \textit{maximal} if there does not exist any $k$-rook configuration in $\cP$, with $k > j$, that properly contains it. The \textit{rook number} $r(\cP)$ is the maximum number of rooks that can be placed in $\cP$ in non-attacking positions. In Figures~\ref{Figure: example rook configuration} (A) and (B), we illustrate a maximal 5-rook configuration and an $r(\cP)$-rook configuration in the same collection $\cP$ of cells, respectively, where $r(\cP) = 7$. \\
We denote by $\cR_j(\cP)$ the set of all $j$-rook configurations of $\cP$ for $j \in \{0, \dots, r(\cP)\}$, with the convention that $\cR_0(\cP)$ represents the empty configuration (i.e., $|\cR_0(\cP)|=1$). Note that the collection $\bigcup_{j=0}^{r(\cP)}\cR_j(\cP)$ forms a simplicial complex, known as the \textit{rook complex} of $\cP$. Two non-attacking rooks in $\cP$ are said to be in \textit{switching position} if they are placed in the diagonal (respectively, anti-diagonal) corner cells of $\cR(I)$, for some inner interval $I$ of $\cP$, and in such a case we say that the rooks are in a \textit{diagonal} (respectively, \textit{anti-diagonal}) position. \\
Fix $j\in \{0,\dots, r(\cP)\}$. Let $F\in \cR_j(\cP)$. Consider two switching rooks, $R_1$ and $R_2$, within $F$, positioned diagonally (or anti-diagonally) in $\cR(I)$ for some inner interval $I$. Let $R_1'$ and $R_2'$ be the rooks in the anti-diagonal (or diagonal, respectively) cells of $\cR(I)$. Then the set $(F\backslash \{R_1, R_2\}) \cup \{R_1', R_2'\}$ also belongs to $\cR_j(\cP)$. This operation of replacing $R_1$ and $R_2$ by $R_1'$ and $R_2'$ is called a \textit{switch of $R_1$ and $R_2$}. This defines an equivalence relation $\sim$ on $\cR_j(\cP)$: $F_1\sim F_2$ if $F_2$ can be obtained from $F_1$ through a sequence of switches. In this case, we say that $F_1$ and $F_2$ are \textit{equivalent with respect to $\sim$} (or \textit{same up to switches}). Figures \ref{Figure: example rook configuration} (C) and (D) show two $4$-rook configurations that are equivalent with respect to $\sim$.\\
Let $\tilde{\cR}_j(\cP) = \cR_j(\cP)/\sim$ be the quotient set of equivalence classes. We define $\widetilde{r}_j(\cP)=\vert \tilde{\cR}_j(\cP)\vert $ for $j\in \{0,\dots, r(\cP)\}$, with the convention that $\widetilde{r}_0(\cP)=1$. The \textit{switching rook polynomial} of $\cP$ is then defined as the polynomial in $\mathbb{Z}_{> 0}[t]$:
$$\widetilde{r}_{\cP}(t)=\sum_{j=0}^{r(\cP)}\widetilde{r}_j(\cP)t^j.$$
For example, for the collection of cells in Figures \ref{Figure: example rook configuration} (C) or (D), $\widetilde{r}_{\cP}(t)=1+8t+19t^2+14t^3+3t^4.$ 

     \begin{figure}[h]
		\centering
            \subfloat[]{\includegraphics[scale=0.7]{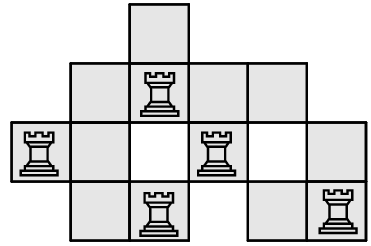}}\quad
            \subfloat[]{\includegraphics[scale=0.7]{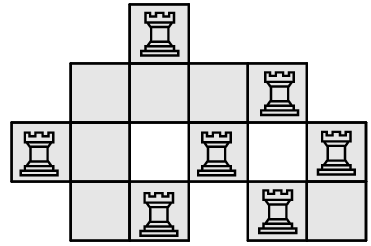}}\quad
		\subfloat[]{\includegraphics[scale=0.7]{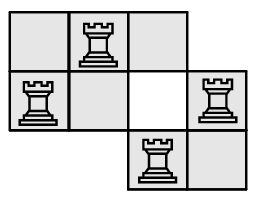}}\quad
            \subfloat[]{\includegraphics[scale=0.7]{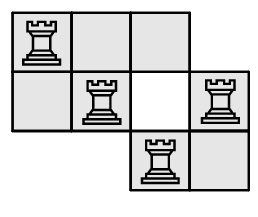}}
            \caption{Examples of some arrangements of non-attacking rooks.}
		\label{Figure: example rook configuration}
    \end{figure} 

        Let $\cQ$ be a rectangle with $w(\cQ)=m$ and $h(\cQ)=n$. After a suitable translation in the plane, we may assume that $\cQ = \cR([(1,1),(m+1,n+1)])$. We denote by $A_{ij}$ the cell of $\cQ$ with lower left corner $(i,j)$, for all $i\in [m]$ and $j \in [n]$. For simplicity, we can identify a rook arrangement with the set of the cells of $\cQ$ where the rooks are placed. We say that a $k$-rook configuration $\fR=\{A_{i_1 j_1},\dots,A_{i_k j_k}\}$ in $\cQ$ is \textit{canonical} if the indices satisfy $i_1 < \dots < i_k$ and $j_1 < \dots < j_k$.

       Observe that for every $k$-rook configuration $\fR$ in $\cQ$ there exists a unique canonical $k$-rook configuration $\fC$ in $\cQ$ such that $\fC\sim \fR$. Indeed, if $\fR=\{A_{i_1 j_1},\dots,A_{i_k j_k}\}$ is a non canonical $k$-rook configuration in $\cQ$, then we have two rooks of $\fR$, namely $A_{i_t j_t}$ and $A_{i_{t'} j_{t'}}$ for some $t,t'\in [k]$, that are in anti-diagonal position. Assume $i_t<i_{t'}$ and $j_t>j_{t'}$. Then one can replace $A_{i_t j_t}$ and $A_{i_{t'} j_{t'}}$ by $A_{i_{t} j_{t'}}$ and $A_{i_{t'} j_{t}}$ in $\fR$, respectively. Repeating the same process, if needed, one obtains the canonical form of $\fR$ after finite number of steps. 

\section{Domino-stable collections of cells}\label{Section: domino stable}

In this section, we introduce a new combinatorial property, called \textit{domino-stability}, which is closely related to the geometric structure of a collection of cells (see Definition \ref{Definition: domino stable}). This property plays a crucial role, as it characterizes precisely those collections of cells whose switching rook polynomial is palindromic.

We begin by introducing a natural construction on axis-aligned rectangles, termed \textit{gluing}, which consists of joining rectangles along parallel edges to form a larger rectangle.

A family of rectangles in $\ZZ^2$ is said to be \textit{axis-aligned} (see Figure \ref{Figure: gluing rectangles} (A)) if there exists a labeling of the rectangles, denoted by $\mathcal{R}_{ij}$ with $i \in [r]$ and $j \in [s]$, such that the following conditions hold:

(i) for each fixed $i \in [r]$, all rectangles $\mathcal{R}_{ij}$ have their lower-left corners in horizontal position  and have the same height, for all $j \in [s]$;

(ii) for each fixed $j \in [s]$, all rectangles $\mathcal{R}_{ij}$ have their lower-left corners in vertical position and have the same width, for all $i \in [r]$.

\begin{Construction}[Gluing of axis-aligned rectangles]\label{Definition: Gluing of rectangles}\rm
   Let $\{\mathcal{R}_{ij}: i \in [r], j \in [s]\}$ be a non-empty family of axis-aligned rectangles and set $$\mathcal{B} = \bigcup_{i \in [r]\atop j \in [s]} \mathcal{R}_{ij}.$$ We denote by $A_{ij}^{(m,n)}$ the $(m,n)$-th cell of $\mathcal{R}_{ij}$, indexed from left to right and from bottom to top (see Figure \ref{Figure: gluing rectangles} (A)). For all $i \in [r]$ and $j \in [s]$, we define $\ell_i = h(\mathcal{R}_{i1})$ and $\omega_j = w(\mathcal{R}_{1j})$. Moreover, we set $\ell = \sum_{p=1}^r \ell_p$ and $\omega = \sum_{q=1}^s \omega_q$. Let $\mathcal{R}$ be a rectangle with cell $B_{uv}$, where $1 \leq u \leq \omega$ and $1 \leq v \leq \ell$, indexed from left to right and from bottom to top. The rectangle $\mathcal{R}$ is said to be the \textit{gluing} of $\cB$, and it is denoted by $\mathrm{G}(\cB)$. \\
   Furthermore, the $(m,n)$-th cell of $\mathcal{R}_{ij}$ is mapped to $B_{uv}$, where $u=m+\sum_{q=1}^{j-1}\omega_q$ and $v=n+\sum_{p=1}^{i-1}\ell_p$, and conversely, $B_{uv}$ maps to $A_{ij}^{(m,n)}$, where $m = u - \sum_{p=1}^{i-1} \ell_p$ and $n = v - \sum_{q=1}^{j-1} \omega_q$. In this case, we say that $A_{ij}^{(m,n)}$ (respectively, $B_{uv}$) \textit{corresponds to} $B_{uv}$ (respectively, $A_{ij}^{(m,n)}$)\textit{, under the gluing,} or that $A_{ij}^{(m,n)}$ and $B_{uv}$ are \textit{corresponding cells under gluing}.

\begin{figure}[h]
		\centering
		\subfloat[$\cB = \cR_{11}\cup \cR_{21}\cup \cR_{31}\cup \cR_{12}\cup \cR_{22}\cup  \cR_{32}$]{\includegraphics[scale=0.65]{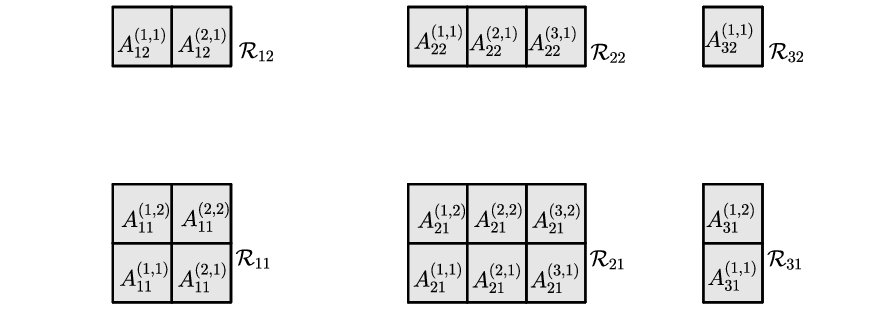}}
            \subfloat[$\cR=\mathrm{G}(\cB)$]{\includegraphics[scale=0.65]{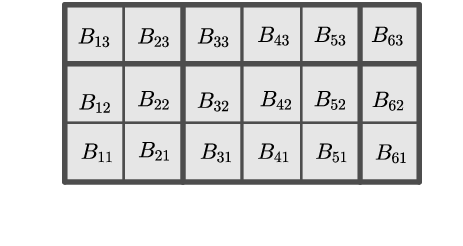}}
            \caption{A family of axis-aligned rectangles and its gluing.}
		\label{Figure: gluing rectangles}
    \end{figure}
\end{Construction}

\begin{Remark}\rm
    Let $\cP$ be a collection of cells. If $\cM(\cP) = \{\cB_1, \dots, \cB_p\}$ denotes the set of all maximal rectangles of $\cP$ and $\cB_i \setminus \left(\bigcup_{j\in [p] \setminus \{i\}} \cB_j\right)\neq \emptyset$, for some $i\in [p]$, then it is easy to see that $\cB_i \setminus \left(\bigcup_{j\in [p] \setminus \{i\}} \cB_j\right)$ consists of a union of axis-aligned rectangles. Therefore, $\mathrm{G}(\overline{\cB_i})$ is well-defined.
\end{Remark}

We are now able to introduce the definition of \textit{domino-stable} collection of cells.

\begin{Definition}[Domino-stable]\label{Definition: domino stable}\rm
    Let $\cP$ be a collection of cells and let $\cM(\cP) = \{\cB_1, \dots, \cB_p\}$ be the set of all maximal rectangles of $\cP$. Set $\overline{\cB_i} = \cB_i \setminus \left(\bigcup_{j\in [p] \setminus \{i\}} \cB_j\right)$, for all $i\in [p]$. We say that $\cP$ is \textit{domino-stable} if, for all $i\in [p]$, the following conditions hold:
        \begin{enumerate}   
            \item if $\overline{\cB_i} \neq \emptyset$, then $\mathrm{G}(\overline{\cB_i})$ is a square;
            \item if $\overline{\cB_i}=\emptyset$, then for every cell $C$ in $\cB_i$ there exist two unique distinct indices $j,k\in [p]\setminus\{i\}$ such that $\overline{\cB_j}$ (respectively, $\overline{\cB_k}$) contains a cell of $\cP$ in the same row (respectively, column) of $C$.
        \end{enumerate}
\end{Definition}

\begin{Remark}\rm\label{Remark: domino stability}
Note that in Definition \ref{Definition: domino stable}, both $\overline{\cB_j}$ and $\overline{\cB_k}$ are non-empty. Consequently, according to condition (1), $\mathrm{G}(\overline{\cB_j})$ and $\mathrm{G}(\overline{\cB_k})$ are squares. 
\end{Remark}

The following example illustrates the Definition \ref{Definition: domino stable}.

\begin{Example}\rm
    In Figures \ref{Figure: example of polyominoes with domino stable and not} (A) and (B), we have two convex polyominoes which are domino-stable. In contrast, Figures \ref{Figure: example of polyominoes with domino stable and not} (C) and (D) show two polyominoes, $\mathcal{P}_1$ and $\mathcal{P}_2$, which do not satisfy this property. The failure occurs because the cell $C$ lies in a maximal rectangle $\mathcal{B}_i$ for which $\overline{\mathcal{B}_i} = \emptyset$, and there is no vertically aligned cell with $C$ that belongs to $\overline{\mathcal{B}}$ for any maximal rectangle of $\mathcal{P}_1$ or $\mathcal{P}_2$.

 \begin{figure}[h]
		\centering
		\subfloat[]{\includegraphics[scale=0.45]{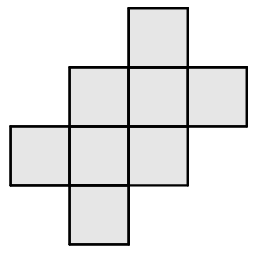}}\quad
            \subfloat[]{\includegraphics[scale=0.45]{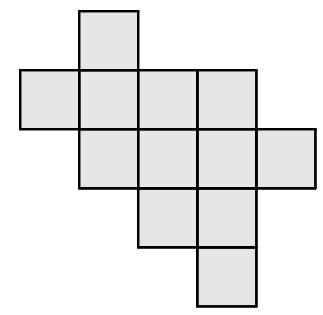}}\quad
            \subfloat[]{\includegraphics[scale=0.45]{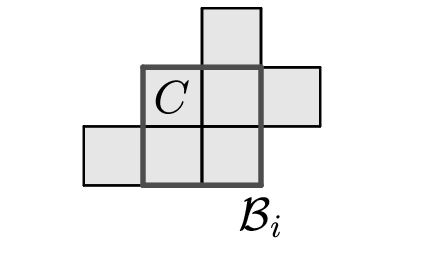}}\quad
            \subfloat[]{\includegraphics[scale=0.45]{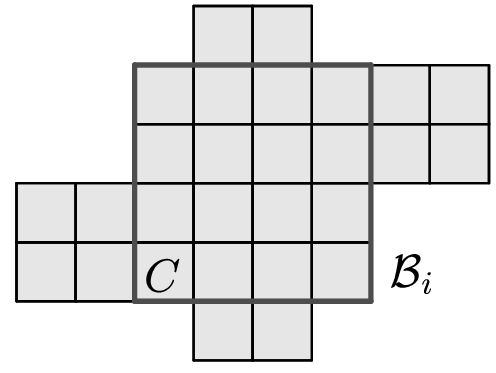}}
            \caption{Two domino-stable polyominoes on the left, two no domino-stable on the right.}
		\label{Figure: example of polyominoes with domino stable and not}
    \end{figure}    
\end{Example}
 
 The next proposition guarantees the existence of a maximal rectangle in a collection of cells whose overline is non-empty and it is fundamental for introducing the so-called \textit{stable-squares} in Definition \ref{Defn: stable square}.

 \begin{Proposition}\label{Proposition: existence of a single square}
	Let $\cP$ be a collection of cells and let $\cM(\cP) = \{\cB_1, \dots, \cB_p\}$ be the set of all maximal rectangles of $\cP$. Then there exists an index $i\in [p]$ such that $\overline{\cB_i}\neq \emptyset$. 
\end{Proposition}
    
\begin{proof}
    Let us assume that $\cP$ is a polyomino. Let $\Rc=[R_1, R_4]$ be the minimal bounding rectangle of $\Pc$, and let $R_2$ and $R_3$ be the anti-diagonal cells of $\Rc$ such that $[R_1, R_4]$ form a horizontal cell interval. Since $\Rc$ is the minimal bounding rectangle of $\Pc$, we obtain that $\Pc \cap [R_1, R_4]\neq \emptyset$. Then, consider a maximal row $[A,B]$ of $\cP$ such that $[A,B]\subseteq [R_1,R_4]$. Due to the maximality of $[A,B]$, there is exactly one maximal rectangle of $\cP$, say $\cB_i$ for some $i\in [p]$, such that $[A,B]\subseteq \cB_i$. Due to the maximality of $\cB_i$, there is a maximal vertical column in $\cB_i$, denoted by $[C,D]$, such that the cell at North of $D$ is not in $\cP$. Moreover, the adjacent cells at East of $A$, West of $B$ and South of $C$ do not belong to $\cP$. Then $C\notin \cB_j$ for all $j\in[p]\setminus \{i\}$, that is, $C\in \overline{\cB_i}$, so $\overline{\cB_i}\neq \emptyset$. Finally, if $\cP$ is a collection of cells, it is enough to consider one connected component, which is a polyomino, and the claim follows from the preceding discussion.
\end{proof}

\begin{Definition}[Stable-square]\rm\label{Defn: stable square}
Let $\cP$ be a collection of cells.  
We set
\[
\cS(\cP)=\{\mathrm{G}(\overline{\cB}): \cB\in\cM(\cP)\text{ and }\overline{\cB}\neq\emptyset\},
\]
which is non-empty by Proposition~\ref{Proposition: existence of a single square}.  
If $\cP$ is domino-stable, then every element of $\cS(\cP)$ is a square; in this case, we refer to the elements of $\cS(\cP)$ as the \emph{stable squares} of $\cP$.
\end{Definition}


Stable squares will play a crucial role in the subsequent results, so we first review a fundamental property of the switching rook polynomial of a square. 

    \begin{Discussion}\rm\label{Discussion: Square}
        Let $\cS$ be a square of size $n^2$, where $n\in \ZZ_{> 0}$. The rook number of $\cS$ is $n$, so let $\widetilde{r}_{\cS}(t) = \sum_{k=0}^{n} \widetilde{r}_k(\cS) t^k$ be the switching rook polynomial of $\cS$. It is well known that $K[\cS]$ is Gorenstein (\cite[pp.451-452]{Sv}) and that the $h$-polynomial of $K[\cS]$ coincides with the switching rook polynomial of $\cS$ (\cite{QRR}). Therefore, by \cite[Theorem 4.1]{S}, we have that $\widetilde{r}_{\cS}(t)$ is palindromic, meaning that $\widetilde{r}_k(\cS) = \widetilde{r}_{n-k}(\cS)$ for all $k \in \{0, \dots, n\}$. This symmetry implies the existence of a bijective correspondence between $\tilde{\cR}_k(\cS)$ and $\tilde{\cR}_{n-k}(\cS)$, so, in what follows, we explicitly construct such a bijection. 
        We may assume that $\cS = \cR([(1,1),(n+1,n+1)])$ and we denote by $A_{ij}$ the cell of $\cS$ with lower left corner $(i,j)$, for all $i,j \in [n]$. Moreover, we refer to $A_{ii}$ for $i\in [n]$ as a \textit{diagonal cell} of $\cS$. 
        For a non-empty canonical $k$-rook configuration $\cA=\{A_{i_1 j_1},\dots,A_{i_k j_k}\}$, we set $[n]\setminus \{i_1 < \dots < i_k\}=\{i_1' < \dots < i_{n-k}'\}$ and $[n]\setminus \{j_1 < \dots < j_k\}=\{j_1' < \dots < j_{n-k}'\}$. Then, we define $\phi_k:\tilde{\cR}_k(\cS)\rightarrow\tilde{\cR}_{n-k}(\cS)$ as:
         \[
                 \phi_k([\cA]_\sim)=\begin{cases} 
                         \left[\{A_{ii}:i\in[n]\}\right]_\sim & \text{if } k=0 \text{ and }\cA=\emptyset, \\
                     \left[\{A_{i_1' j_1'},\dots,A_{i_{n-k}' j_{n-k}'}\}\right]_\sim& \text{if } 0< k\leq r(\cS) \text{ and }\cA=\{A_{i_1 j_1},\dots,A_{i_k j_k}\}.
                             \end{cases}         
         \]
         It is an easy exercise to verify that this map is well-defined and bijective. In Figure \ref{Figure: Case square, C and phi(C)}, we illustrate that, if $\cA = \{A_{12}, A_{64}, A_{76}\}$, then $\phi_3([\cA]_\sim) = [\{A_{21}, A_{33}, A_{45}, A_{57}, A_{88}\}]_\sim$. We provide an intuitive explanation of how the construction of $\phi_k([\cA]_\sim)$ proceeds, as this may help the reader grasp the underlying philosophy of the arguments throughout the remainder of the paper.
         \begin{enumerate} 
          \item Given a $k$-rook configuration in $\cS$, consider its canonical arrangement namely $\cA$. 
          \item Remove the rows and columns of $\cS$ that contain a rook of $\cA$ and glue the remaining rows and columns of $\cS$ to obtain a square $\cS'$ of size $n-k$. 
          \item Place the rooks on the diagonal cells of $\cS'$ and map these rooks back into $\cS$ by placing them in the corresponding cells under the gluing. 
         \end{enumerate}
         
    \begin{figure}[h]
		\centering
		\includegraphics[scale=0.65]{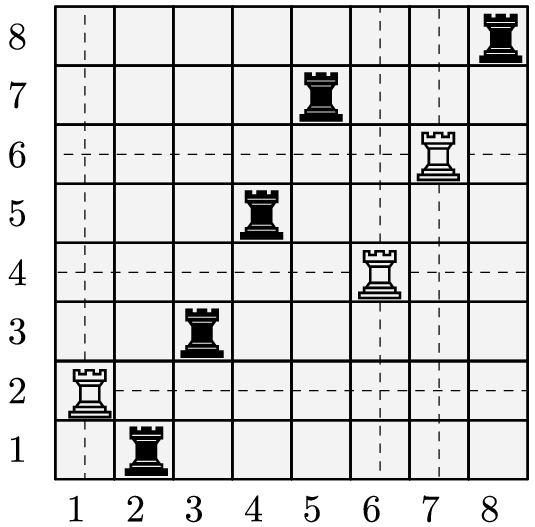}
            \caption{Rooks of $\cA$ (white) and of the representative of $\phi_3([\cA]_\sim)$ (black).}
		\label{Figure: Case square, C and phi(C)}
    \end{figure}
    
    \end{Discussion}

    We conclude this section with the following result, which will be used throughout the paper. It formulates condition (2) of domino-stability for the cells in $\cB_i \setminus \overline{\cB_i}$, when $\overline{\cB_i}$ is non-empty.

\begin{Proposition}\label{Proposition: A cell in Bi minus bar(Bi) has the condition 2}
    Let $\cP$ be a domino-stable collection of cells. Let $\cM(\cP) = \{\cB_1, \dots, \cB_p\}$ be the set of all maximal rectangles of $\cP$ and $\overline{\cB_i} = \cB_i \setminus \left(\bigcup_{j\in [p] \setminus \{i\}} \cB_j\right)$, for $i\in [p]$, with $\overline{\cB_i}\neq \emptyset$. Then for every cell $C$ of $\cB_i\setminus \overline{\cB_i}$ there exist two unique and distinct indices $j,k\in [p]$ such that $\overline{\cB_j}$ (respectively, $\overline{\cB_k}$) intersects the row of $\cP$ (respectively, the column) containing $C$.
\end{Proposition}
    
\begin{proof}
   Consider $C \in \cB_i \setminus \overline{\cB_i}$. Let $H$ be the row of $\cP$ containing $C$, and let $\cB_j$ be the maximal rectangle in $\cM(\cP)$ containing $H$. We claim $\overline{\cB_j} \neq \emptyset$. Suppose, for contradiction, that $\overline{\cB_j} = \emptyset$. Then, by condition (2) of Definition \ref{Definition: domino stable}, there exists a unique index $h \in [p] \setminus \{j\}$ such that $\overline{\cB_h}$ contains a cell of $H$; let us denote this cell by $D$. Since $\overline{\cB_h} = \cB_h \setminus \left(\bigcup_{\alpha \in [p] \setminus \{h\}} \cB_\alpha\right)$, it follows that $D$ cannot be in $\cB_j$, contradicting the fact that $D \in H \subseteq \cB_j$. Hence, $\overline{\cB_j} \neq \emptyset$. We now show that there exists a cell of $H$ which belongs to $\overline{\cB_j}$. Since $\overline{\cB_j} \neq \emptyset$, we may consider a cell in $\overline{\cB_j}$, that we denote by $E$. If $E\in H$, then we are done. If $E\notin H$, then consider the column of $\cP$ which contains $E$, that we indicate by $V$. We denote by $A$ the unique cell in $H\cap V$ and we claim that $A\in \overline{\cB_j}$. In fact, $V$ is not contained in $\cB_\ell$, for all $\ell \neq i$, otherwise $E$ would not belong to $\overline{\cB_j}$; the same conclusion holds for $H$. Therefore  $A\in \overline{\cB_j}$. To prove the uniqueness of the index $j$, suppose, for the sake of contradiction, that there exists an index $j' \in [p] \setminus \{i, j\}$ such that $\overline{\cB_{j'}}$ contains a cell $A'$ of $\cP$ in horizontal position with $C$, so $A'\in H$. Since $H \subseteq \cB_j$ and $\overline{\cB_{j'}} = \cB_{j'} \setminus \left(\bigcup_{\beta \in [p] \setminus \{j'\}} \cB_\beta\right),$ then $A'$ cannot be in $\overline{\cB_{j'}}$, which is a contradiction. The existence and uniqueness of the index $k \in [p]$ such that $\overline{\cB_k}$ contains a cell of $\cP$ in vertical position with $C$ can be proven by applying similar arguments as before. This completes the proof of the proposition.
\end{proof}

\section{Domino-stability implies palindromicity}\label{Section3: sufficient}

The aim of this section is to prove that domino-stability provides a sufficient condition for the switching rook polynomial to be palindromic. Before establishing this result, we present some preliminary results.

\begin{Lemma}\label{Lemma: every row and column intersects a unique overline}
Let $\cP$ be a domino-stable collection of cells, and let $\cM(\cP)=\{\cB_1,\ldots,\cB_p\}$ be the set of its maximal rectangles. 

1) For every row $H$ of $\cP$, there exists a unique index $i\in[p]$ such that $H\cap\overline{\cB_i}\neq\emptyset$ and $H\subseteq\cB_i$. Likewise, for every column $V$ of $\cP$, there exists a unique index $k\in[p]$ such that $V\cap\overline{\cB_k}\neq\emptyset$ and $V\subseteq\cB_k$.

2) For a cell $C$ of $\cP$, let $i$ and $k$ be the indices in 1) determined by the row and column containing $C$, respectively. Then:
\begin{itemize}
    \item $C$ belongs to $\overline{\cB_i}$ if and only if $i=k$
    \item If $i\ne k$, then $w(\cB_k)<w(\cB_i)$ and $h(\cB_i)<h(\cB_k)$.
\end{itemize}

\end{Lemma}

\begin{proof}
1) Let $H$ be a row of $\cP$ and $\cB_i$ be the maximal rectangle of $\cP$ such that $H\subseteq\cB_i$. We first claim that $\overline{\cB_i}\neq\emptyset$. Suppose, by contradiction, that $\overline{\cB_i}=\emptyset$, and choose a cell $C\in H$. By condition~(2) of Definition~\ref{Definition: domino stable}, there exist an index $j\in[p]\setminus\{i\}$ and a cell $D\in\overline{\cB_j}$ lying in the same row of $\cP$ as $C$. Since $H$ is the row of $\cP$ containing $C$, we have $D\in H\subseteq\cB_i$. This contradicts $D\in\overline{\cB_j}$, because a cell of $\overline{\cB_j}$ does not belong to any maximal rectangle different from $\cB_j$. Hence $\overline{\cB_i}\neq\emptyset$.

Choose a cell $C\in H$. If $C\in\overline{\cB_i}$, then $H\cap\overline{\cB_i}\neq\emptyset$. Suppose instead that $C\in\cB_i\setminus\overline{\cB_i}$. By Proposition~\ref{Proposition: A cell in Bi minus bar(Bi) has the condition 2}, there exist an index $j\in[p]$ and a cell $D\in\overline{\cB_j}$ lying in the same row of $\cP$ as $C$. Thus $D\in H\subseteq\cB_i$. Since $D\in\overline{\cB_j}$, this is possible only if $j=i$. Therefore, also in this case, $H\cap\overline{\cB_i}\neq\emptyset$.

It remains to prove uniqueness. Suppose that $H\cap\overline{\cB_j}\neq\emptyset$ for some $j\in[p]$, and choose a cell $D\in H\cap\overline{\cB_j}$. Since $H\subseteq\cB_i$, we have $D\in\cB_i$. The definition of $\overline{\cB_j}$ then gives $j=i$. Hence $\overline{\cB_i}$ is the unique set among $\overline{\cB_1},\ldots,\overline{\cB_p}$ that intersects $H$.

The assertion for the columns follows by interchanging rows and columns.

\medskip

2) Let $C$ be a cell of $\cP$, and $H$ and $V$ be the row and column of $\cP$ containing $C$. By 1), $i$ and $k$ are indices in $[p]$ such that $H\subseteq\cB_i$ and $V\subseteq\cB_k$. 

If $i=k$, then $H,V\subseteq\cB_i$, and $\cB_i$ is the only maximal rectangle containing $C$. Hence $C\in\overline{\cB_i}$. Conversely, if $C\in\overline{\cB_i}$, then $V$ intersects $\overline{\cB_i}$, so the uniqueness of the index associated with $V$ gives $k=i$.

Suppose now that $i\neq k$. For every $j\in[p]$ with $C\in\cB_j$, we denote by $H_j$ and $V_j$ the horizontal and the vertical cell interval of $\cB_j$ containing $C$, respectively; note that $H_j$ and $V_j$ are a row and a column \emph{of $\cB_j$}, in general not of $\cP$, and that $|H_j|=w(\cB_j)$ and $|V_j|=h(\cB_j)$. From $H_k\subseteq H=H_i$ and $V_i\subseteq V=V_k$ we obtain $$w(\cB_k)=|H_k|\leq|H_i|=w(\cB_i) \text{ and } h(\cB_i)=|V_i|\leq|V_k|=h(\cB_k).$$

We claim that both inequalities are strict. Suppose that $w(\cB_k)=w(\cB_i)$. Since \noindent $H_k\subseteq H_i$ and these two rows contain the same number of cells, we have $H_k=H_i$. Together with $V_i\subseteq V_k$, this implies, by the definition of a rectangle, that $\cB_i\subseteq\cB_k$. Since $\cB_i$ and $\cB_k$ are distinct maximal rectangles, this is impossible. Therefore, $w(\cB_k)<w(\cB_i).$ Similarly, if $h(\cB_i)=h(\cB_k)$, then $V_i=V_k$. Together with $H_k\subseteq H_i$, this would imply that $\cB_k\subseteq\cB_i$, again contradicting the maximality of the two distinct rectangles. Consequently, $h(\cB_i)<h(\cB_k).$
\end{proof}

We now describe the arrangement of a maximum number of non-attacking rooks in a domino-stable collection of cells, as well as its uniqueness up to switching operations.

\begin{Lemma}\label{Prop: unique way to place non-attacking rooks}
   Let $\cP$ be a domino-stable collection of cells. Let $\cM(\cP) = \{\cB_1, \dots, \cB_p\}$ and $\cS(\cP) = \{\cS_1, \dots, \cS_q\}$ be the sets of maximal rectangles and stable squares of $\cP$, respectively, where $q \leq p$. Then there exists a unique $r(\cP)$-rook configuration of $\cP$, up to switches, and $$r(\cP)=\sum_{i=1}^q \sqrt{\mathrm{size}(\cS_i)}.$$ 
\end{Lemma}

\begin{proof}
    After suitable labeling of the maximal rectangles of $\cP$, we may assume that $\overline{\cB_i} \neq \emptyset$, for all $1\leq i\leq q$. Moreover, we set $\cS_i = \mathrm{G}(\overline{\cB_i})$ and $d_i = \sqrt{\mathrm{size}(\cS_i)}$, for each $i \in [q]$. Let $\cT_i$ denote the arrangement of $d_i$ rooks, placed on the cells of $\overline{\cB_i}$, which correspond under the gluing to the diagonal cells of $\cS_i$, for all $i\in [q]$ (see Discussion \ref{Discussion: Square}). Finally, define $\cT = \cT_1 \sqcup \dots \sqcup \cT_q$ and set $d = \sum_{i=1}^q d_i = \vert \cT \vert$. We now prove that $\cT$ provides the unique way to place the maximum number of non-attacking rooks on $\cP$, up to switches.
    
    1) We first show that $\cT$ is a $d$-rook configuration in $\cP$. Let $i \in [q]$. Since each rook in $\cT_i$ is placed in a cell of $\overline{\cB_i}$, and by definition $\overline{\cB_i} = \cB_i \setminus \bigcup_{j \neq i} \cB_j$, it follows that all rooks in $\cT_i$ are in non-attacking positions with those in $\cT_j$ for all $j \neq i$. Moreover, by definition of $\cT_i$, every rook in $\cT_i$ is in non-attacking position with a rook in $\cT_i$. Hence, $\cT$ is a $d$-rook configuration in $\cP$.
    
    2) Next, we prove that $\cT$ is maximal. Suppose, by contradiction, that there exists a $k$-rook configuration $\cT'$ in $\cP$ with $k>d$ such that $\cT\subset \cT'$. It is not restrictive to assume that $k=d+1$, that is, $\cT'=\cT \cup \{R\}$, where $R$ is a rook not in $\cT$. Let $C_R$ be the cell of $\cP$ where $R$ is placed and $\cB_i$ be a maximal rectangle of $\cP$ containing $C_R$, for some $i \in [p]$. Since $\cP$ is domino-stable, the following holds.
    \begin{itemize}
    \item If $C_R\in \overline{\cB_i}$, then $\overline{\cB_i} \neq \emptyset$, so $\mathrm{G}(\overline{\cB_i})$ is a square by (1) of Definition \ref{Definition: domino stable} and all the columns and rows of $\overline{\cB_i}$ are occupied by a rook of $\cT$. Consequently, $R$ is in an attacking position with a rook in $\cT$, a contradiction.
    \item If $\overline{\cB_i} \neq \emptyset$ and $C_R\in \cB_i\setminus \overline{\cB_i}$, then by Proposition \ref{Proposition: A cell in Bi minus bar(Bi) has the condition 2} there exist two cells $A$ and $B$ of $\cP$ in horizontal and vertical position with $C_R$, and only two distinct indices $j,k\in [p]$ such that $A \in \overline{\cB_{j}}$ and $B \in \overline{\cB_{k}}$. Hence $R$ is necessarily in an attacking position with a rook of $\cT$, which is a contradiction
    \item On the other hand, if $\overline{\cB_i} = \emptyset$, then a similar previous contradiction arises from (2) of Definition \ref{Definition: domino stable}. 
    \end{itemize} 
    In each case, we get a contradiction; therefore, we can conclude that $\cT$ is maximal.

3) We now show that $k\leq d$ for every $k$-rook configuration $\cU$ of $\cP$, and that $\cU\sim\cT$ whenever $k=d$. 
Let $\cU$ be a $k$-rook configuration of $\cP$. For all $i\in[q]$, let $\cU_i$ be the set of rooks of $\cU$ placed on cells of $\overline{\cB_i}$ and let $\cQ$ be the set of the remaining rooks of $\cU$, so that $\cU=\left(\bigsqcup_{i\in[q]}\cU_i\right)\sqcup\cQ$. Set $r_i=|\cU_i|$ and $r=\sum_{i\in[q]}r_i$, hence $k=r+|\cQ|$.

Consider a rook $R\in\cQ$ and denote by $C_R$ the cell of $\cP$ where $R$ is placed. Let $i(R)$ and $j(R)$ be the indices attached by 1) of Lemma~\ref{Lemma: every row and column intersects a unique overline} to the row and to the column of $C_R$, respectively; equivalently, by the previous paragraph, $i(R)$ and $j(R)$ are the unique indices such that the row of $R$ contains a rook of $\cT_{i(R)}$ and the column of $R$ contains a rook of $\cT_{j(R)}$. Since $C_R$ does not belong to $\overline{\cB_i}$ for any $i\in[q]$, 2) of Lemma~\ref{Lemma: every row and column intersects a unique overline} gives $i(R)\neq j(R)$ and $w(\cB_{j(R)})<w(\cB_{i(R)})$.
For all $i\in[q]$, set $a_i=|\{R\in\cQ:i(R)=i\}|$ and $b_i=|\{R\in\cQ:j(R)=i\}|$. Since each rook of $\cQ$ is counted exactly once in each of the two families, we have $\sum_{i\in[q]}a_i=\sum_{i\in[q]}b_i=|\cQ|$.

Fix $i\in[q]$. The rooks of $\cU_i$ lie on $r_i$ distinct rows among the $d_i$ rows of $\cP$ meeting $\overline{\cB_i}$, and by definition every rook $R\in\cQ$ with $i(R)=i$ lies on one of those $d_i$ rows as well. Since the rooks of $\cU$ are non-attacking, no two of them lie on a same row of $\cP$, and therefore $r_i\leq d_i$ and $a_i\leq d_i-r_i$. Arguing on the columns instead of the rows, we also get $b_i\leq d_i-r_i$. In particular,
\[
k=r+|\cQ|=r+\sum_{i\in[q]}a_i\leq r+\sum_{i\in[q]}(d_i-r_i)=d.
\]
Assume now that $k=d$ and suppose, by contradiction, that $\cQ\neq\emptyset$. Then all the above inequalities are equalities, that is, $a_i=b_i=d_i-r_i$ for all $i\in[q]$. Consequently,
\[
0=\sum_{i\in[q]}w(\cB_i)(a_i-b_i)=\sum_{R\in\cQ}\bigl(w(\cB_{i(R)})-w(\cB_{j(R)})\bigr)>0,
\]
which is a contradiction. Hence $\cQ=\emptyset$, that is, all the rooks of $\cU$ are placed on cells of $\bigcup_{i\in[q]}\overline{\cB_i}$, and from $\sum_{i\in[q]}r_i=d=\sum_{i\in[q]}d_i$ together with $r_i\leq d_i$ we obtain $r_i=d_i$ for all $i\in[q]$. Finally it remains to observe that $\cU\sim\cT$ 
by Discussion~\ref{Discussion: Square}. 

    In conclusion, $\cT$ yields the unique $r(\cP)$-rook configuration of $\cP$, up to switches, thus completing the proof of our claim, since $r(\cP) = \vert \cT \vert = \sum_{i=1}^q \sqrt{\mathrm{size}(\cS_i)}$.
    \end{proof}

We are now ready to prove that the switching rook polynomial of a domino-stable collection of cells is always palindromic. The key idea is that, starting from the $r(\cP)$-rook configuration described in the Lemma \ref{Prop: unique way to place non-attacking rooks}, any configuration of $k$ rooks can be transformed, up to switches, into a configuration of $(r(\cP)-k)$ non-attacking rooks, and vice versa. 

 \begin{Proposition}\label{Proposition: domino stable implies palindromic}
       Let $\cP$ be a collection of cells and $\widetilde{r}_{\cP}(t)=\sum_{j=0}^d\widetilde{r}_j(\cP)t^j$ be the switching rook polynomial of $\cP$, where $d=r(\cP)$. If $\cP$ is domino-stable then $\widetilde{r}_{\cP}(t)$ is palindromic. 
\end{Proposition}

\begin{proof}
We keep the notation introduced in the proof of Lemma~\ref{Prop: unique way to place non-attacking rooks}: after a suitable relabeling of $\cM(\cP)=\{\cB_1,\ldots,\cB_p\}$, we have $\overline{\cB_i}\neq\emptyset$ if and only if $i\in[q]$, and we set $\cS_i=\mathrm{G}(\overline{\cB_i})$ and $d_i=\sqrt{\mathrm{size}(\cS_i)}$ for all $i\in[q]$. Moreover, $\cT=\cT_1\sqcup\cdots\sqcup\cT_q$, where $\cT_i$ is the arrangement of $d_i$ rooks placed on the cells of $\overline{\cB_i}$ which correspond under the gluing to the diagonal cells of $\cS_i$, is the unique $d$-rook configuration of $\cP$ up to switches, where $d=r(\cP)=\sum_{i\in[q]}d_i$.

We first observe that every row and every column of $\cP$ contains exactly one rook of $\cT$ by Lemma~\ref{Lemma: every row and column intersects a unique overline}.
Write $\cT=\{T_1,\ldots,T_d\}$ and let $H_a$ and $V_a$ denote, respectively, the row and the column of $\cP$ containing $T_a$, for all $a \in [d]$. Thus $\{H_1,\ldots,H_d\}$ and $\{V_1,\ldots,V_d\}$ are, respectively, the sets of all rows and all columns of $\cP$.

Whenever $H_a\cap V_b\neq\emptyset$, denote by $C_{ab}$ the unique cell in this intersection. In particular, $T_a$ is placed on $C_{aa}$.

Let $\cR$ be a $k$-rook configuration of $\cP$. For every rook $R\in\cR$, let $C_R$ be the cell occupied by $R$. Since every row and every column of $\cP$ contains exactly one rook of $\cT$, there exist unique indices $s(R),t(R)\in[d]$ such that $T_{s(R)}$ lies in the same row as $R$ and $T_{t(R)}$ lies in the same column as $R$, that is, $C_R=H_{s(R)}\cap V_{t(R)}=C_{s(R)t(R)}.$ We associate with $R$ the arrow
\[
s(R)\longrightarrow t(R),
\]
where $s(R)$ and $t(R)$ are called \textit{source} and \textit{target} of the arrow. In particular, note that $C_{aa}=H_a\cap V_a$ is the cell occupied by $T_a$, and hence an arrow $a\to a$ occurs if and only if $\cR$ contains the rook $T_a$. We now define the following set:
\[
A(\cR)=\{s(R)\to t(R):R\in\cR\}.
\]
The configuration $\cR$ is uniquely determined by $A(\cR)$, since an arrow $a\to b$ corresponds to the rook placed on $C_{ab}$. Since the rooks of $\cR$ are non-attacking, every index of $[d]$ is the source of at most one arrow and the target of at most one arrow. Indeed, two arrows with the same source would correspond to two rooks in the same row of $\cP$, whereas two arrows with the same target would correspond to two rooks in the same column of $\cP$.

Starting from any arrow and successively following its target, there is therefore at most one arrow that can follow it. Similarly, there is at most one arrow that can precede it. Since $[d]$ is finite, this procedure produces either a maximal directed path
\[
a_0\longrightarrow a_1\longrightarrow\cdots\longrightarrow a_m,
\]
where $a_0$ is the target of no arrow and $a_m$ is the source of no arrow, or a directed cycle
\[
a_1\longrightarrow a_2\longrightarrow\cdots\longrightarrow a_m\longrightarrow a_1.
\]
The indices occurring in different paths or cycles are disjoint. An index which is neither the source nor the target of an arrow is regarded as a path of length zero, whereas an arrow $a\to a$ is regarded as a cycle of length one.

We now describe the switch associated with two consecutive arrows. Consider $$a\longrightarrow b, b\longrightarrow c \in A(\cR),$$ where $a\neq b$ and $b\neq c$. 
Let $i\in[q]$ be the unique index such that $T_b\in\cT_i$. Since $T_b$ is placed on $C_{bb}\in\overline{\cB_i}$, both the row $H_b$ and the column $V_b$ intersect $\overline{\cB_i}$. By part~(1) of Lemma~\ref{Lemma: every row and column intersects a unique overline}, we have $H_b\subseteq\cB_i$ and $V_b\subseteq\cB_i.$ Now $C_{ab}\in V_b$ and $C_{bc}\in H_b$, and hence $C_{ab},C_{bc}\in\cB_i$. Since $\cB_i$ is a rectangle, the minimal bounding rectangle $\mathcal F$ of $C_{ab}$ and $C_{bc}$ is contained in $\cB_i$, and therefore $\mathcal F\subseteq\cP.$ Moreover, $C_{ab}$ and $C_{bc}$ lie in distinct rows and distinct columns, so they are opposite corner cells of $\mathcal F$. The remaining two corner cells lie, respectively, in $H_a\cap V_c$ and $H_b\cap V_b$. They are therefore $C_{ac}$ and $C_{bb}$. Consequently, replacing the rooks on $C_{ab}$ and $C_{bc}$ with rooks on $C_{ac}$ and $C_{bb}$ is a valid switch:
\[
\{C_{ab},C_{bc}\}
\sim
\{C_{ac},C_{bb}\}.
\]
In terms of arrows, this switch replaces
\[
a\longrightarrow b,\qquad b\longrightarrow c
\]
with
\[
a\longrightarrow c,\qquad b\longrightarrow b.
\]
Thus the two consecutive arrows are replaced by an arrow joining their initial and terminal indices and by the loop corresponding to the rook $T_b$.

The previous arguments can be easily extended via induction in order to get the following claims.

\textbf{Claim 1.} If $a_0\longrightarrow a_1\longrightarrow\cdots\longrightarrow a_m$ 
is a maximal path of $A(\cR)$, with $m\geq 1$, then $C_{a_0a_m}$ is a cell of $\cP$ and the rooks associated with the arrows of this path can be replaced, by a sequence of switches, with the rooks placed on $C_{a_0a_m},C_{a_1a_1},\ldots,C_{a_{m-1}a_{m-1}}.$ In other words,
\[\{C_{a_0a_1},C_{a_1a_2},\ldots,C_{a_{m-1}a_m}\}
\sim
\{C_{a_0a_m},T_{a_1},\ldots,T_{a_{m-1}}\}.
\]


\textbf{Claim 2.} If $a_1\longrightarrow a_2\longrightarrow\cdots\longrightarrow a_m\longrightarrow a_1
$
is a directed cycle of $A(\cR)$, then
\[
\{C_{a_1a_2},C_{a_2a_3},\ldots,C_{a_{m-1}a_m},C_{a_ma_1}\}
\sim
\{T_{a_1},\ldots,T_{a_m}\}.
\]

Let $\cR$ be a $k$-rook configuration and we now define from $\cR$ a new $(d-k)$-rook configuration, which we will denote by $\cR^\vee$. Let us observe that the directed graph with vertex set $[d]$ and arrow set $A(\cR)$ is a disjoint union of maximal directed paths and directed cycles. Here a maximal directed path may have length zero: in this case it consists of a single index $a\in[d]$ which is neither the source nor the target of an arrow of $A(\cR)$. The rules to define $\cR^\vee$ are the following:
\begin{itemize}
    \item If $
a_0\longrightarrow a_1\longrightarrow\cdots\longrightarrow a_m$ (with $m\geq 0$) is a maximal directed path, then we place one rook on the cell $C_{a_0a_m}$. \\
Observe that, if $m\geq 1$, the fact that $C_{a_0a_m}$ is a cell of $\cP$ follows from Claim 1. If $m=0$, then the path consists of an isolated index $a_0$. 
\item For a directed cycle, we do not place any rook. 
\end{itemize} 

From the above definition it is clear that $\cR^\vee$ is a configuration of non-attacking rooks. 
We next determine the number of rooks of $\cR^\vee$. A maximal path containing $s$ indices has $s-1$ arrows and contributes one rook to $\cR^\vee$, while a directed cycle containing $s$ indices has $s$ arrows and contributes no rook. Since the maximal paths and the directed cycles partition $[d]$, and since $A(\cR)$ has exactly $k$ arrows, the difference between the total number of indices and the total number of arrows is precisely the number of maximal paths. Consequently, $|\cR^\vee|=d-k$, that is, $\cR^\vee\in\mathcal R_{d-k}(\cP).$

Let $\cR,\cU\in\mathcal{R}_k(\cP)$ be such that $\cU$ is obtained from $\cR$ by a single switch. We claim that $\cR^\vee\sim(\cU)^\vee$. 
Let $\cU$ be obtained from $\cR$ by switching the rooks placed on $C_{ab}$ and $C_{ce}$, with $a\neq c$ and $b\neq e$. The other two corner cells of the rectangle in which the switch is performed are $C_{ae}$ and $C_{cb}$. Therefore, $$
A(\cU)
=
\bigl(A(\cR)\setminus\{a\to b,c\to e\}\bigr)
\cup
\{a\to e,c\to b\}.$$

We compare the maximal paths and directed cycles of $\cA(\cR)$ and $\cA(\cU)$ before and after this replacement. We distinguish the following cases.
\medskip

 (1) Suppose first that the two arrows $a\to b$ and $c\to e$ belong to the same maximal path. After interchanging their roles if necessary, this path has the form
\[
s\longrightarrow\cdots\longrightarrow a\longrightarrow b\longrightarrow\cdots\longrightarrow c\longrightarrow e\longrightarrow\cdots\longrightarrow t.
\]
After the switch, it is replaced by the maximal path
\[
s\longrightarrow\cdots\longrightarrow a\longrightarrow e\longrightarrow\cdots\longrightarrow t
\]
and the directed cycle
\[
b\longrightarrow\cdots\longrightarrow c\longrightarrow b.
\]
The initial and terminal indices of the maximal path remain $s$ and $t$, while the new directed cycle contributes no rook to $\cU^\vee$. Thus these indices contribute the same rook $C_{st}$ to both $\cR^\vee$ and $\cU^\vee$.
\medskip

(2) Suppose that one arrow belongs to a maximal path and the other to a directed cycle. We may write them as
\[
s\longrightarrow\cdots\longrightarrow a\longrightarrow b\longrightarrow\cdots\longrightarrow t
\]
and
\[
c\longrightarrow e\longrightarrow\cdots\longrightarrow c.
\]
After the switch, the path and the cycle are joined into the maximal path
\[
s\longrightarrow\cdots\longrightarrow a\longrightarrow e\longrightarrow\cdots\longrightarrow c\longrightarrow b\longrightarrow\cdots\longrightarrow t.
\]
Its initial and terminal indices are again $s$ and $t$, so the rook contributed to $\cU^\vee$ is the same rook $C_{st}$ contributed to $\cR^\vee$.
\medskip

(3) If the two arrows belong to directed cycles, the switch either joins two distinct cycles into one directed cycle or divides one directed cycle into two directed cycles. Since directed cycles contribute no rooks to the construction, we again obtain the same rooks in $\cR^\vee$ and $\cU^\vee$.
\medskip

(4) It remains to consider the case in which the two arrows belong to two distinct maximal paths. Write these paths as
\[
s_1\longrightarrow\cdots\longrightarrow a\longrightarrow b\longrightarrow\cdots\longrightarrow t_1,\qquad
s_2\longrightarrow\cdots\longrightarrow c\longrightarrow e\longrightarrow\cdots\longrightarrow t_2.
\]
After the switch, the two resulting maximal paths are
\[
s_1\longrightarrow\cdots\longrightarrow a\longrightarrow e\longrightarrow\cdots\longrightarrow t_2, \qquad
s_2\longrightarrow\cdots\longrightarrow c\longrightarrow b\longrightarrow\cdots\longrightarrow t_1.
\]
Therefore, the original paths contribute to $\cR^\vee$ the rooks placed on $C_{s_1t_1}$ and $C_{s_2t_2}$ whereas the new paths contribute to $\cU^\vee$ the rooks placed on $C_{s_1t_2}$  and $C_{s_2t_1}.$ We show that these two pairs of rooks are in switching position.

By assumption, $\cU$ is obtained from $\cR$ by switching the rooks placed on $C_{ab}$ and $C_{ce}$. Hence, by the definition of a switch, there exists a rectangle $\mathcal D\subseteq\cP$ having $C_{ab}$ and $C_{ce}$ as one pair of opposite corner cells and $C_{ae}$ and $C_{cb}$ as the other pair. In particular, $\mathcal D$ meets the rows $H_a$ and $H_c$ and the columns $V_b$ and $V_e$.

We first establish two elementary facts that will allow us to construct a rectangle containing the cells associated with the initial and terminal indices of the two paths.

\begin{itemize}
    \item[(F1)] Suppose that $\mathcal E\subseteq\cP$ is a rectangle, that $\mathcal E\cap H_v\neq\emptyset$, and that $u\to v$ belongs to $A(\cR)$. We claim that there exists a rectangle $\mathcal E'\subseteq\cP$ such that $\mathcal E\subseteq\mathcal E'$ and $\mathcal E'\cap H_u\neq\emptyset.$\\
Indeed, let $i\in[q]$ be the index such that $T_v\in\cT_i$. Since $T_v$ is placed on $C_{vv}\in\overline{\cB_i}$, Lemma~\ref{Lemma: every row and column intersects a unique overline} gives $H_v\subseteq\cB_i$ and $V_v\subseteq\cB_i.$ Choose a cell $F\in\mathcal E\cap H_v$. Since $H_v\subseteq\cB_i$, we have $F\in\cB_i$. Moreover, the arrow $u\to v$ corresponds to the cell $C_{uv}\in H_u\cap V_v.$ As $V_v\subseteq\cB_i$, we also have $C_{uv}\in\cB_i$. Since $\cB_i$ is a rectangle containing both $F$ and $C_{uv}$, it contains the unique cell $F'$ which lies in $H_u$ and is in vertical position with $F$. Let $\mathcal E'$ be the minimal bounding rectangle of $\mathcal E\cup\{F'\}$. The cells of $\mathcal E$ lying in $H_v$ form a row of $\mathcal E$ contained in $\cB_i$. Hence every cell added to $\mathcal E$ in order to obtain $\mathcal E'$ lies in $\cB_i$. Therefore, $\mathcal E'\subseteq\mathcal E\cup\cB_i\subseteq\cP.$ Furthermore, $F'\in\mathcal E'\cap H_u$, proving our claim.

\item[(F2)] The analogous assertion holds for columns. Namely, if $\mathcal E\subseteq\cP$ is a rectangle, $\mathcal E\cap V_v\neq\emptyset$, and $v\to w$ belongs to $A(\cR)$, then there exists a rectangle $\mathcal E''\subseteq\cP$ such that $\mathcal E\subseteq\mathcal E''
$ and $\mathcal E''\cap V_w\neq\emptyset.$ 
\end{itemize}

We now apply~(F1) to the first path. Starting from $\mathcal D$, which meets $H_a$, and proceeding backwards through the arrows of $s_1\longrightarrow\cdots\longrightarrow a,$ we obtain a rectangle contained in $\cP$ which contains $\mathcal D$ and meets $H_{s_1}$. Indeed, if $s_1=a_0\longrightarrow a_1\longrightarrow\cdots\longrightarrow a_r=a,$ we first apply~(F1) to $a_{r-1}\to a_r$, then to $a_{r-2}\to a_{r-1}$, and continue until we reach $H_{a_0}=H_{s_1}$. Applying the same argument to $s_2\longrightarrow\cdots\longrightarrow c,$
we obtain a rectangle $\mathcal D'$ contained in $\cP$, containing $\mathcal D$, and meeting both $H_{s_1}$ and $H_{s_2}$. Since $\mathcal D\subseteq\mathcal D'$, the rectangle $\mathcal D'$ still meets $V_b$ and $V_e$. We now apply~(F2) successively to the arrows of $b\longrightarrow\cdots\longrightarrow t_1$ in their natural order, obtaining a rectangle which meets $V_{t_1}$. We then apply the same argument to $e\longrightarrow\cdots\longrightarrow t_2.$
In this way, we obtain a rectangle $\mathcal E\subseteq\cP$ satisfying $\mathcal D\subseteq\mathcal E$ and meeting the two rows $H_{s_1},H_{s_2}$ and the two columns $V_{t_1},V_{t_2}$. Since $\mathcal E$ is a rectangle, for every $\lambda,\mu\in\{1,2\}$ it contains the cell lying in the row $H_{s_\lambda}$ and in the column $V_{t_\mu}$. This cell is precisely $C_{s_\lambda t_\mu}$. Hence $\mathcal E$ contains $C_{s_1t_1},C_{s_1t_2},C_{s_2t_1}$ and $C_{s_2t_2}.$ 
The indices $s_1$ and $s_2$ are distinct, and the indices $t_1$ and $t_2$ are distinct, because the two maximal paths are disjoint. Therefore, $C_{s_1t_1}$ and $C_{s_2t_2}$ are opposite corner cells of a rectangle contained in $\mathcal E$, whose other two corner cells are $C_{s_1t_2}$ and $C_{s_2t_1}$. This rectangle is contained in $\cP$, and hence $\{C_{s_1t_1},C_{s_2t_2}\}
\sim
\{C_{s_1t_2},C_{s_2t_1}\}.
$ All the other rooks of $\cR^\vee$ and $\cU^\vee$ coincide. Consequently, $\cR^\vee\sim\cU^\vee.$
\medskip

We have proved that one switch from $\cR$ to $\cU$ induces equivalent configurations $\cR^\vee$ and $\cU^\vee$. By applying this argument to every switch in a sequence, we obtain that, if $\cR\sim\cU$ then $\cR^\vee\sim\cU^\vee$. Hence the map
\[
\rho_k:\widetilde{\mathcal R}_k(\cP)\longrightarrow\widetilde{\mathcal R}_{d-k}(\cP),
\qquad
\rho_k([\cR]_\sim)=[\cR^\vee]_\sim,
\]
is well-defined for every $k\in\{0,\ldots,d\}$. 

We now prove that $\cR\sim (\cR^\vee)^\vee$. A non-trivial maximal path $a_0\longrightarrow a_1\longrightarrow\cdots\longrightarrow a_m
$
related to $\cR$ gives in $A(\cR^\vee)$ the single arrow
$a_0\to a_m$, while $a_1,\ldots,a_{m-1}$ become isolated indices.
Hence this path contributes to $(\cR^\vee)^\vee$ the rooks placed on $C_{a_0a_m},C_{a_1a_1},\ldots,C_{a_{m-1}a_{m-1}}$, which are equivalent to the original rooks of the path by Claim~1.\\
If
$
a_1\longrightarrow\cdots\longrightarrow a_m\longrightarrow a_1
$
is a directed cycle related to $\cR$, then its indices are isolated in
$A(\cR^\vee)$. They therefore contribute the rooks
$T_{a_1},\ldots,T_{a_m}$ to $(\cR^\vee)^\vee$, and these are
equivalent to the original rooks of the cycle by Claim~2. Finally, an
isolated index for $\cR$ becomes a loop in $A(\cR^\vee)$ and hence
contributes no rook to $(\cR^\vee)^\vee$. Hence, we can conclude that $(\cR^\vee)^\vee\sim\cR.$

Therefore, $\rho_{d-k}\circ\rho_k
=
\operatorname{id}_{\widetilde{\mathcal R}_k(\cP)}
$ for every $k\in\{0,\ldots,d\}$. Replacing $k$ with $d-k$, we also obtain 
$
\rho_k\circ\rho_{d-k}
=
\operatorname{id}_{\widetilde{\mathcal R}_{d-k}(\cP)}.
$
Thus $\rho_k$ is a bijection and $\widetilde r_k(\cP)=\widetilde r_{d-k}(\cP)$ for every $k\in\{0,\ldots,d\}$, that is, $\widetilde r_{\cP}(t)$ is palindromic.
\end{proof}

\section{The converse and the proof of the main theorem}\label{Section4: necessary}

    Just as the previous section established a sufficient condition for the palindromicity, we now show that domino-stability is also a necessary condition for a switching rook polynomial to be palindromic. Before presenting the main result, we state a preliminary lemma.

    \begin{Lemma}\label{Lemma: every column and row has a rook}
        Let $\cP$ be a collection of cells, $\cM(\cP) = \{\cB_1, \dots, \cB_p\}$ be the set of maximal rectangles of $\cP$ and $\widetilde{r}_{\cP}(t)=\sum_{j=0}^d\widetilde{r}_j(\cP)t^j$ be the switching rook polynomial of $\cP$, where $d=r(\cP)$. Suppose that $\widetilde{r}_{d}(\cP)=1$ and let $\cT$ be a $d$-rook configuration in $\cP$. Then:
        \begin{enumerate}
            \item  every column of $\cP$ contains a cell occupied by a rook of $\cT$, and the same holds for every row of $\cP$;
            \item if $R$ is a rook placed in a cell of $\cP$, that does not contain any rook of $\cT$, then there exist two unique rooks in $\cT$ that are in attacking position with $R$.
        \end{enumerate}
       
    \end{Lemma}
    
    \begin{proof}
        (1) Assume by contradiction that $C$ is a column of $\cP$ whose cells do not have any rook of $\cT$. Now, if there were a row $L$ of $\cP$ that intersects $C$ and whose cells do not have any rook of $\cT$, then we could consider $\cT \cup \{F\}$, where $F$ is a rook placed in the cell $C \cap L$; this would be a contradiction with the maximality of $\cT$. Therefore, there must be a rook $T$ in a cell of every row of $\cP$ that intersects $C$. Now, select a rook $T'$ in a row $L'$, and consider $(\cT \setminus \{T\}) \cup \{T'\}$, where $T'$ is placed in the cell $C \cap L'$. This provides a $r(\cP)$-rook configuration in $\cP$ that is not equivalent to $\cT$, contradicting the assumption $\widetilde{r}_d(\cP) = 1$. A similar argument shows that every row of $\cP$ contains a cell occupied by a rook of $\cT$.
        
        (2) Suppose that there are no two rooks of $\cT$ that are in attacking position with $R$. Indeed, if no rook of $\cT$ is in an attacking position with $R$, then $\cT\cup\{R\}$ would be a set of non-attacking rooks, which contradicts the maximality of $\cT$. Furthermore, suppose that there exists no rook of $\cT$ which is in the same row (respectively, column) of $R$, but $R$ is in the same column (respectively, row) of a rook $T$ of $\cT$. In this case, replacing $T$ with $R$ yields a new configuration $\cT' = (\cT\setminus \{T\})\sqcup\{R\} \in {\cR}_d(\cP)$. Since $\cT'$ is obtained from $\cT$ by swapping a rook in the same column (respectively, row), then $\cT$ and $\cT'$ are not equivalent with respect to $\sim$, contradicting the assumption that $\widetilde{r}_d(\cP) = 1$. Finally, to prove uniqueness, it is enough to observe that we cannot have three rooks of $\cT$ in the attacking position with $R$, since this would imply that $\cT$ contains two attacking rooks, which contradicts the definition of $\cT$.
    \end{proof}

     \begin{Proposition}\label{Proposition: P has NO the domino stable then r_1 less than rd-1}
        Let $\cP$ be a collection of cells and $\widetilde{r}_{\cP}(t)=\sum_{j=0}^d\widetilde{r}_j(\cP)t^j$ be the switching rook polynomial of $\cP$, where $d=r(\cP)$. Suppose that $\widetilde{r}_d(\cP)=1$ and $\cP$ is not domino-stable. Then $\widetilde{r}_1(\cP)<\widetilde{r}_{d-1}(\cP).$     
    \end{Proposition}

\begin{proof}
 Let $\cM(\cP) = \{\cB_1, \dots, \cB_p\}$ be the set of maximal rectangles of $\cP$.  Since $\widetilde{r}_d(\cP) = 1$, there exists a unique way to place $r(\cP)$ non-attacking rooks on $\cP$, up to switches. By the preceding lemma, every row and every column of $\cP$ contains exactly one rook of $\cT$. In particular, $\cP$ has exactly $d$ rows and $d$ columns. Let $[\cT]_\sim \in \tilde{\cR}_d(\cP)$.  Note that $\widetilde{r}_1(\cP) = \vert\cP\vert$. Let us begin by easily establishing an injection $\tau$ from $\tilde{\cR}_1(\cP)$ to $\tilde{\cR}_{d-1}(\cP)$. If $R$ is a rook of $\cT$, then $\cT\setminus \{R\} \in {\cR}_{d-1}(\cP)$. If $R \notin \cT$, by Lemma \ref{Lemma: every column and row has a rook} (2), there exists a unique pair of rooks of $\cT$, namely $T_H(R)$ and $T_V(R)$, that are in the same row and column of $R$, respectively. We now define the map $\tau: \tilde{\cR}_1(\cP) \to \tilde{\cR}_{d-1}(\cP)$ by
        \[
                \tau([\{R\}]_\sim)=\begin{cases} 
                    [\cT\setminus \{R\}]_{\sim} & \text{if } R\in \cT, \\
                    {\left[\left(\cT\setminus \left\{T_H^{(R)}, T_V(R)\right\}\right)\sqcup \{R\}\right]}_\sim & \text{if } R\notin \cT.
                            \end{cases}         
        \]
        for every $[R]_\sim \in \tilde{\cR}_1(\cP)$. It is easy to verify that $\tau$ is well-defined and injective. Therefore, $\widetilde{r}_1(\cP)\leq \widetilde{r}_{d-1}(\cP)$. 
        
    Now, we show that there exists $[\cR]_\sim \in \tilde{\cR}_{d-1}(\cP)$ such that $[\cR]_\sim \notin \tau(\tilde{\cR}_1(\cP))$. For this aim we need the following claim.


\medskip
\noindent\textbf{Claim.}
Suppose that there exist a maximal rectangle $\cB_i$ of $\cP$ and
a row $H$ of $\cP$ such that $H\subseteq\cB_i$ and the rook of
$\cT$ contained in $H$ is placed on a cell
$
D\in\cB_i\setminus\overline{\cB_i}.
$
Then $\tau$ is not surjective. The same conclusion holds if rows and
columns are interchanged.

\smallskip
\noindent\emph{Proof of the claim.}
For what follows, the reader can refer to Figure \ref{fig:obstruction-tau}.
\begin{figure}[h]
\centering
\begin{tikzpicture}[
    x=1cm,
    y=1cm,
    cell/.style={
        draw=black,
        line width=0.7pt,
        fill=gray!16
    },
    distinguished/.style={
        draw=black,
        line width=0.7pt,
        fill=white
    },
    boundary/.style={
        draw=gray!75,
        line width=1.1pt,
        dashed
    }
]

\draw[boundary] (0,0) rectangle (4,3);
\node[anchor=south west] at (0.08,0.06) {$\cB_i$};

\draw[cell] (1,0) rectangle ++(1,1);
\node at (1.5,0.5) {$R_1$};

\draw[distinguished] (1,1) rectangle ++(1,1);
\node at (1.5,1.5) {$A$};

\draw[cell] (2,1) rectangle ++(1,1);
\node at (2.5,1.5) {$R_2$};

\draw[distinguished] (2,3) rectangle ++(1,1);
\node at (2.5,3.5) {$B$};

\draw[cell] (3,3) rectangle ++(1,1);
\node at (3.5,3.5) {$R_3$};

\node[gray!70,font=\Large] at (1.5,3.5) {$\times$};

\end{tikzpicture}
\caption{
A representation of the rooks $R_1, R_2, R_3$ and of the cells $A,B$.
}
\label{fig:obstruction-tau}
\end{figure}
Set $H_2=H$, denote by $R_2$ the rook of $\cT$ placed on $D$, and let $V_2$ be the column containing $R_2$. Since $D\notin\overline{\cB_i}$, there exists a maximal rectangle $\cB_j\neq\cB_i$ containing $D$. The horizontal cell interval of $\cB_j$ containing $D$ is contained in $H$, and therefore in $\cB_i$. Since $\cB_j$ is maximal and distinct from $\cB_i$, it cannot be contained in $\cB_i$. 
It follows that $V_2$ contains a cell $B\in\cP\setminus\cB_i$ having an edge in common with a cell of $\cB_i$. Let $H_3$ be the row of $\cP$ containing $B$. By the preceding lemma, $H_3$ contains a rook $R_3$ of $\cT$. Since $B\in V_2$ and $R_2$ already lies in $V_2$, the rook $R_3$ is not placed on $B$.
Due to the maximality of $\cB_i$, there exists a cell $A\in H$ such that the cell in horizontal position with $B$ and in vertical position with $A$ does not belong to $\cP$. 
Let $V_1$ be the column of $\cP$ containing $A$, and let $R_1$ be the rook of $\cT$ lying in $V_1$. The rooks $R_1,R_2,R_3$ are pairwise distinct. Moreover, $H_2\cap V_1=\{A\}$ (where $H_2=H$), $H_3\cap V_2=\{B\}$ and $H_3\cap V_1=\emptyset$. Let $R_A$ and $R_B$ be rooks placed on $A$ and $B$, respectively, and set \[
\cU=
\bigl(\cT\setminus\{R_1,R_2,R_3\}\bigr)
\sqcup
\{R_A,R_B\}.\]
One can easily verify that $[\cU]_\sim\notin\tau\bigl(\widetilde{\mathcal R}_1(\cP)\bigr),$ and hence $\tau$ is not surjective. This proves the claim. 

\medskip

It remains to prove that, whenever $\cP$ is not domino-stable, the
hypothesis of the claim is satisfied. We now distinguish two cases, which involve negating conditions (1) and (2) of Definition \ref{Definition: domino stable}, respectively.

\medskip

\textit{Case 1.}
Suppose that condition~(1) of Definition~\ref{Definition: domino stable} fails. Then there exists $i\in[p]$ such that
$
\overline{\cB_i}\neq\emptyset
$
and $\mathrm G(\overline{\cB_i})$ is not a square. After interchanging rows and columns if necessary, we may assume that $
m=w\bigl(\mathrm G(\overline{\cB_i})\bigr)<h\bigl(\mathrm G(\overline{\cB_i})\bigr)=n.
$
By the definition of the gluing, exactly $n$ rows of $\cP$ and exactly $m$ columns of $\cP$ intersect $\overline{\cB_i}$. Every row intersecting $\overline{\cB_i}$ is contained in $\cB_i$. Indeed, if a row $H$ contains a cell of $\overline{\cB_i}$, then every maximal rectangle containing $H$ also contains that cell and must therefore be $\cB_i$. By the preceding lemma, each of these $n$ rows contains exactly one rook of $\cT$, and these $n$ rooks lie in distinct columns. If the column containing each of these rooks intersected $\overline{\cB_i}$, then the $n$ rooks would lie in $n$ distinct columns among the $m$ columns intersecting $\overline{\cB_i}$, which is impossible because $m<n$. Therefore, there exists a row $H$ intersecting $\overline{\cB_i}$ such that the rook $R_2$ of $\cT$ lying in $H$ is placed on a cell $D$ whose column does not intersect $\overline{\cB_i}$. In particular, $D\in\cB_i\setminus\overline{\cB_i}.$
The Claim therefore applies.

\medskip

\textit{Case 2.}
Suppose that condition~(2) of
Definition~\ref{Definition: domino stable} fails. Then there exist
an index $\ell\in[p]$ and a cell $C\in\cB_\ell$ such that $\overline{\cB_\ell}=\emptyset$
and condition~(2) is not satisfied at $C$. Let $H$ and $V$ be the row and the column of $\cP$ containing $C$, respectively. We first observe that every row of $\cP$ intersects at most one of
the sets $\overline{\cB_1},\ldots,\overline{\cB_p}$. Indeed, suppose that a row $L$ intersects $\overline{\cB_j}$, and let $\cB_h$ be a
maximal rectangle containing $L$. Since $\cB_h$ contains a cell of
$\overline{\cB_j}$, we have $\cB_h=\cB_j$. Hence $L\subseteq\cB_j,$
and the index $j$ is uniquely determined. The analogous assertion
holds for the columns of $\cP$.

Suppose that $H$ intersects $\overline{\cB_j}$ and that $V$
intersects $\overline{\cB_k}$. Then $j$ and $k$ are uniquely
determined, and $j,k\neq\ell$. We claim that $j\neq k$. Otherwise, $H\subseteq\cB_j
$ and $V\subseteq\cB_j.$ The horizontal and vertical cell intervals of $\cB_\ell$ containing
$C$ are therefore contained in $\cB_j$. Since $\cB_j$ is a
rectangle, it follows that $\cB_\ell\subseteq\cB_j.$ The maximality of $\cB_\ell$ yields $\cB_\ell=\cB_j$, contradicting $\overline{\cB_\ell}=\emptyset$ and $\overline{\cB_j}\neq\emptyset.$ Thus $j\neq k$.

Since condition~(2) fails at $C$, at least one between $H$ and $V$ does not intersect
any of $\overline{\cB_1},\ldots,\overline{\cB_p}$. By the column analogue of the preceding claim,
we may assume that $H\cap\overline{\cB_j}=\emptyset$ for every $j\in[p]$.
Let $\cB_i$ be a maximal rectangle containing $H$, and let $D$ be
the cell occupied by the unique rook $R_2$ of $\cT$ contained in
$H$. Then $H\subseteq\cB_i$ and $D\in\cB_i.$ In particular, $H\cap\overline{\cB_i}=\emptyset,$
and hence $D\notin\overline{\cB_i}$. Therefore, $D\in\cB_i\setminus\overline{\cB_i},$ so the hypotheses of the Claim are satisfied and the Claim applies.
\medskip

In both cases, $\tau$ is not surjective. In conclusion, we have that $\widetilde r_1(\cP)
<
\widetilde r_{d-1}(\cP),
$
as desired.
\end{proof}

     \begin{Proposition}\label{Proposition: palindromic implies P has the domino stable}
        Let $\cP$ be a collection of cells and $\widetilde{r}_{\cP}(t)=\sum_{j=0}^d\widetilde{r}_j(\cP)t^j$ be the switching rook polynomial of $\cP$, where $d=r(\cP)$. If $\widetilde{r}_{\cP}(t)$ is palindromic, then $\cP$ is domino-stable.     
    \end{Proposition}

    \begin{proof}
             We proceed by contradiction, assuming that $\cP$ is not domino-stable. Since $\widetilde{r}_{\cP}(t)$ is palindromic, then $\widetilde{r}_d(\cP) = 1$. By Proposition \ref{Proposition: P has NO the domino stable then r_1 less than rd-1}, it must hold that $\widetilde{r}_{1}(\cP) < \widetilde{r}_{d-1}(\cP)$, a contradiction. Therefore, $\cP$ is domino-stable, as claimed.
    \end{proof}

We can finally present the main theorem of the paper, which provides a complete characterization of the palindromicity of switching rook polynomials in terms of the geometry of the related collections of cells.
    \begin{Theorem}\label{Thm: main palindromicity}
Let $\cP$ be a collection of cells and $\widetilde{r}_{\cP}(t)$ be the switching rook polynomial of $\cP$. Then, $\widetilde{r}_{\cP}(t)$ is palindromic if and only if $\cP$ is domino-stable.
\end{Theorem}

\begin{proof}
    This result follows directly from Propositions \ref{Proposition: domino stable implies palindromic} and \ref{Proposition: palindromic implies P has the domino stable}.
\end{proof}

Theorem~\ref{Thm: main palindromicity}, together with
Proposition~\ref{Proposition: P has NO the domino stable then r_1 less than rd-1},
yields the following numerical criterion: once the leading coefficient
is equal to one, palindromicity is determined by the first and
next-to-leading coefficients.

\begin{Corollary}\label{Corollary: numerical palindromicity criterion}
Let $\cP$ be a collection of cells, let $d=r(\cP)$, and assume that $\widetilde r_d(\cP)=1.$
Then
\[
\widetilde r_1(\cP)\leq \widetilde r_{d-1}(\cP),
\]
Moreover, the following conditions are equivalent:
\begin{enumerate}
    \item $\widetilde r_1(\cP)=\widetilde r_{d-1}(\cP)$;
    \item $\cP$ is domino-stable;
    \item $\widetilde r_{\cP}(t)$ is palindromic.
\end{enumerate}
\end{Corollary}

\begin{proof}
If $\cP$ is not domino-stable, then Proposition~
\ref{Proposition: P has NO the domino stable then r_1 less than rd-1}
gives $\widetilde r_1(\cP)<\widetilde r_{d-1}(\cP).$ Consequently, $
\widetilde r_1(\cP)=\widetilde r_{d-1}(\cP)
$
implies that $\cP$ is domino-stable. Conversely, if $\cP$ is
domino-stable, then Theorem~\ref{Thm: main palindromicity} shows that
$\widetilde r_{\cP}(t)$ is palindromic, and hence $\widetilde r_1(\cP)=\widetilde r_{d-1}(\cP).
$
The equivalence between the second and third conditions is precisely
Theorem~\ref{Thm: main palindromicity}. The asserted inequality follows
from the same two alternatives.
\end{proof}

\section{Algebraic consequences and an open problem}\label{Section: last, main theorem}

Let $\cP$ be a collection of cells, and let $S_\cP=K[x_\mathbf{v} \mid \mathbf{v}\in V(\cP)]$ be the polynomial ring associated with $\cP$, where $K$ is a field. We recall that a proper interval $[\mathbf{a},\mathbf{b}]$ of $\ZZ^2$ is called an \textit{inner interval} of $\cP$ if all cells in $\cR([\mathbf{a},\mathbf{b}])$ belong to $\cP$. If $[\mathbf{a},\mathbf{b}]$ is an inner interval of $\cP$, with diagonal corners $\mathbf{a}$ and $\mathbf{b}$ and anti-diagonal corners $\mathbf{c}$ and $\mathbf{d}$, then the binomial $x_\mathbf{a}x_\mathbf{b}-x_\mathbf{c}x_\mathbf{d}$ is called an \textit{inner $2$-minor} of $\cP$. The ideal $I_{\cP}$ of $S_\cP$ generated by all inner $2$-minors of $\cP$ is called the \textit{ideal of $2$-minors} of $\cP$. The quotient ring $K[\cP] = S_\cP/I_{\cP}$ is called the \textit{coordinate ring} of $\cP$. 

Let $h_{K[\cP]}(t)$ denote its
$h$-polynomial, that is, the numerator of the reduced Hilbert-Poincar\'e series of $K[\cP]$.  For several classes of collections of cells, it is known
that $
h_{K[\cP]}(t)=\widetilde{r}_{\cP}(t);
$
see, for instance, \cite{CNU3,DN2,JN,KV,NQR,QRR,RR,WG}. The following consequence of
Theorem~\ref{Thm: main palindromicity} applies whenever this equality
holds.

\begin{Corollary}\label{Coro: Gorenstein}
Let $\cP$ be a collection of cells such that $h_{K[\cP]}(t)=\widetilde{r}_{\cP}(t).$
Then the following statements hold:
\begin{enumerate}
    \item If $K[\cP]$ is Gorenstein, then $\cP$ is domino-stable.
    \item If $K[\cP]$ is a Cohen--Macaulay domain, then $K[\cP]$ is
    Gorenstein if and only if $\cP$ is domino-stable.
\end{enumerate}
\end{Corollary}

\begin{proof}
Suppose first that $K[\cP]$ is Gorenstein. By
\cite[Theorem~4.1]{S}, the $h$-polynomial of $K[\cP]$ is palindromic.
Since $h_{K[\cP]}(t)=\widetilde{r}_{\cP}(t),$
the switching rook polynomial of $\cP$ is palindromic. Therefore,
Theorem~\ref{Thm: main palindromicity} implies that $\cP$ is
domino-stable.

Suppose now that $K[\cP]$ is a Cohen--Macaulay domain. The forward
implication follows from part~(1). Conversely, assume that $\cP$ is
domino-stable. By Theorem~\ref{Thm: main palindromicity},
$\widetilde{r}_{\cP}(t)$ is palindromic. It follows from the assumed equality $h_{K[\cP]}(t)=\widetilde{r}_{\cP}(t)
$
that the $h$-polynomial of $K[\cP]$ is palindromic. Hence
\cite[Theorem~4.4]{S} yields that $K[\cP]$ is Gorenstein.
\end{proof}

For simple thin polyominoes and parallelogram polyominoes, the
preceding result recovers \cite[Theorem~4.2]{RR} and
\cite[Theorem~4.10]{QRR}, respectively. More generally,
Corollary~\ref{Coro: Gorenstein} applies to every class of collections of
cells for which the equality $h_{K[\cP]}(t)=\widetilde{r}_{\cP}(t)
$
and the Cohen--Macaulay domain property are known.

The Cohen--Macaulay domain assumption in
Corollary~\ref{Coro: Gorenstein} is essential for the application of
Stanley's criterion. Nevertheless, computational experiments reported in
\cite{N2} suggest that domino-stability may be sufficient for the
Gorenstein property even without this assumption. This motivates the
following conjecture.

\begin{Conjecture}\label{Conj Gorenstein}
Let $\cP$ be a collection of cells. If $\cP$ is domino-stable, then
$K[\cP]$ is Gorenstein.
\end{Conjecture}

\begin{footnotesize}

{\bf Acknowledgments.} The second and third authors are supported by Scientific and Technological Research Council of Turkey T\"UB\.{I}TAK under the Grant No: 124F113, and are thankful to T\"UB\.{I}TAK for their support. The first author is supported by Scientific and Technological Research Council of Turkey T\"UB\.{I}TAK under the Grants 126F175 and 124F113, respectively. The first and third authors are members of INDAM-GNSAGA and acknowledge its support. The authors are grateful to the HPC group at Sabancı University for installing \texttt{Macaulay2} \cite{M2} on the Tosun cluster and for providing access to such a valuable computational resource. Moreover, the authors would like to thank D. Grayson, M. Stillman, and D. Torrance for their personal communications and insightful suggestions.

{\bf Declaration of competing interest.}  The authors declare that they have no known competing financial interests or personal relationships that could have appeared to influence the work reported in this paper.

{\bf Data availability.} Data used for this article are provided in \cite{N2}.
\end{footnotesize}


	\end{document}